\documentclass[12pt,twoside,reqno,psamsfonts]{amsart}

\usepackage[colorlinks = true, citecolor = black, linkcolor = black, urlcolor = black, pdfstartview=FitH]{hyperref}  

\usepackage{graphicx}
\usepackage{amsmath}
\usepackage{amssymb}
\usepackage{enumerate}
\usepackage{verbatim}
\usepackage{url}

\numberwithin{equation}{section}
\newcommand{\R}{\mathbb{R}}
\newcommand{\C}{\mathbb{C}}
\newcommand{\Q}{\mathbb{Q}}
\newcommand{\N}{\mathbb{N}}

\newcommand{\E}{\mathbb{E}}

\newcommand{\cB}{\mathcal{B}}

\newcommand{\cL}{\mathcal{L}}

\newcommand{\cZ}{\mathcal{Z}}

\newcommand{\cF}{\mathcal{F}}
\newcommand{\cE}{\mathcal{E}}

\newcommand{\cR}{\mathcal{R}}

\newcommand{\cW}{\mathcal{W}}

\newcommand{\cA}{\mathcal{A}}

\newcommand{\eps}{\varepsilon}

\newcommand{\dist}{\operatorname{dist}}

\newcommand{\Hd}{\dim_\mathrm{H}}

\newcommand{\Z}{\mathbb{Z}}

\renewcommand{\emptyset}{\varnothing}
\renewcommand{\epsilon}{\varepsilon}
\renewcommand{\rho}{\varrho}
\renewcommand{\phi}{\varphi}

\renewcommand{\mod}{\,\,\mathrm{mod\,}}

\renewcommand{\a}{\mathbf{a}}
\renewcommand{\b}{\mathbf{b}}

\newcommand{\A}{\mathbf{A}}

\renewcommand{\iint}{\int\hspace{-0.1in}\int}

\DeclareMathOperator{\spt}{spt}

\DeclareMathOperator{\supp}{supp}

\theoremstyle{plain}

\newtheorem{theorem}{Theorem}[section]

\newtheorem{problem}[theorem]{Problem}

\theoremstyle{definition}

\theoremstyle{remark}
\newtheorem{remark}[theorem]{Remark}

\usepackage[usenames,dvipsnames]{xcolor}

\usepackage{subfigure}

\graphicspath{ {pictures/} }

\usepackage[left=2.22cm,right=2.22cm,top=2.4cm,bottom=2.4cm]{geometry}

\begin{document}


\title{Fourier transforms and iterated function systems}

\author{Tuomas Sahlsten}
\address{Department of Mathematics and Statistics, University of Helsinki, Finland   \& \newline
Department of Mathematics, University of Manchester, Manchester, United Kingdom}
\email{tuomas.sahlsten@helsinki.fi}

\dedicatory{Dedicated to the memory of my father Heikki Sahlsten}


\begin{abstract}
We discuss the problem of bounding the Fourier transforms of stationary measures of iterated function systems (IFSs) and how the \textit{pseudo-randomness} of the IFS either due to arithmetic, algebraic or geometric reasons is reflected in the behaviour of the Fourier transform. We outline various methods that have been built to estimate the Fourier transform of stationary measures arising e.g. from thermodynamical formalism, additive combinatorics, random walks on groups and hyperbolic dynamics. Open problems, prospects and recent links to quantum chaos are also highlighted.
\end{abstract}

\maketitle


\section{Introduction}

An \textit{iterated function system} (IFS) in $\R^d$, $d \geq 1$, is defined typically (e.g. in Falconer's book \cite{Falconer}) as a collection $\Phi = \{f_j : j \in \cA\}$, where $f_j : \R^d \to \R^d$ are contractions and $\cA$ is either a finite or infinite countable set. They arise naturally in various settings, see e.g. Ghosh's survey \cite{GhoshSurvey}, such as from codings of Anosov flows and chaotic billiards using Markov partitions and Poincar\'e sections (see e.g. \cite{Bowen,BowenSeries, AGY}), where they describe the long-time statistics of the chaotic dynamics of the flow. Any IFS $\Phi$ is always associated an \textit{attractor} $K_\Phi$ satisfying $K_\Phi = \bigcup_{j \in \cA} f_j (K_\Phi)$. We can then study the statistics of typical orbits under the IFS $\Phi$ with the help of \textit{stationary (Gibbs) measures} $\mu_\phi$ on $K_\Phi$ associated to a family of \textit{potentials} $\phi = \{\phi_j : j \in \cA\}$, $\phi_j : \R^d \to \C$ (see Bowen \cite{BowenGibbs} in 1975 and Fan and Lau \cite{FanLau} from 1999 in this form). These are defined as probability measures satisfying the relation $\int u \, d\mu_\phi = \int \cL_\phi u \, d\mu_\phi$ for all continuous and compactly supported $u : \R^d \to \C$, where $\cL_\phi$ is the \textit{transfer operator} defined by 
$$\cL_\phi u(x) := \sum_{ j \in \cA } e^{\phi_j(f_j(x))} u(f_j(x)), \quad x \in \R^d.$$ 
Gibbs measures allow us to construct many natural measures on $K_\Phi$ such as conformal measures, measures of maximal entropy and \textit{Bernoulli measures}, that is, those with $\phi_j(f_j(x)) \equiv \log p_j$ for all $x$ for some $0 < p_j < 1$ with $\sum_{j \in \cA} p_j = 1$. For IFSs Bernoulli measures are particularly well-studied since they enjoy strong statistical limit laws due to close connections to i.i.d. random walks. For similitudes $f_j$ Bernoulli measures are commonly called \textit{self-similar}, for affine $f_j$ as \textit{self-affine} and for conformal $f_j$ as \textit{self-conformal} \cite{Falconer}. Much progress over the recent years has been put into understanding the \textit{dimension theory} of these measures, such as describing their Hausdorff dimension with thermodynamic formulae (e.g. Hutchinson's formula \cite{Hut} from 1981). This can be very difficult when the maps exhibit \textit{overlaps} in $\R$ or non-commutative behaviour of the linear parts in the higher dimensional cases (see e.g. breakthroughs by Hochman from 2014 \cite{Hochman}, Shmerkin from 2019 \cite{Shmerkin}, Varj\'u from 2019 \cite{VarjuAnnals} and the surveys by Varj\'u \cite{Varju,Varju2}).

In this manuscript we would like to discuss about the recent developments in the problem of relating the \textit{Fourier transform} $\widehat{\mu} : \R^d \to \C$,
$$\widehat{\mu}(\xi) := \int e^{-2\pi i \xi \cdot x} \, d\mu(x), \quad \xi \in \R^d,$$
 of the stationary measures $\mu$ to the \textit{ergodic theoretic properties} of the underlying IFS $\Phi$. This is a kind of question that also arises in the field of \textit{quantum chaos}, where one likes to understand how the classical chaotic dynamics influences the behaviour of the quantization of this classical dynamics (e.g. how mixing properties of the geodesic flows on surfaces influence the quantum dynamics defined using Fourier integral operators (see e.g. works of Shnirelman, Colin de Verdiere, and Zelditch \cite{Sni,CdV,Zelditch} from 1974, 1985 and 1987 and Rudnick-Sarnak \cite{RS94} from 1994). Understanding the Fourier decay properties of $\mu$ can reveal some of these mechanisms such as via the recently developed \textit{Fractal Uncertainty Principles} in quantum chaos after the works of Dyatlov and Zahl \cite{DyatlovZahl} in 2016 and Bourgain and Dyatlov \cite{BD1} in 2017, where Fourier transforms of $\mu$ play an important role, see the survey \cite{Dyatlov} for a wider overview of the links and related problems in quantum chaos. In the literature, those $\mu$ whose Fourier transform satisfies $\widehat{\mu}(\xi) \to 0$ as $|\xi| \to \infty$, is called \textit{Rajchman measure}, see the earlier survey by \cite{Lyons} from 1995.

There is also appeal in pure mathematics to study Fourier transforms of stationary measures $\mu$ arising from IFSs. Information about the asymptotic behaviour of $\widehat{\mu}$ provides a method to study the uniqueness of trigonometric series in the support of $\spt \mu$ as discussed in  \cite{Lyons,Salem2,K,SZ,Zygmund}, prevalence of normal numbers in $\spt \mu$ since the work of Davenport, Erd\"os and LeVeque \cite{DEL} from 1963 and numbers with Diophantine properties by the work of Pollington, Velani, Zafeiropoulos and Zorin \cite{PVZZ} in 2022, Marstrand type projection theorems (see the book by Mattila \cite{Mattila95} from 1995 for an overview), constructing patterns in $\spt \mu$ in the work of Laba and Pramanik \cite{LP} in 2009, how the intersections of the supports of $\spt \mu$ with other sets behave (see e.g. recent work of Avila, Lyubich and Zhang \cite{ALZ}), Fourier multipliers in the work of Sarnak \cite{Sarnak} from 1980 \cite{Sarnak} and Sidorov and Solomyak \cite{SidSol} from 2003, conditional decay of correlations in the work of Wormell \cite{Wormell} from 2023, exponential mixing in the work of Li and Pan \cite{LiPan} from 2023 (see also \cite{LPX}), restriction problems in the work of Mockenhaupt \cite{Mo} from 2000 or a method to help in studying the absolute continuity of Bernoulli convolutions in the work of Shmerkin \cite{ShmerkinConv} from 2014. We will discuss a bit more the emerged applications in this survey while we go around the topics more in detail, but I will point out to the articles \cite{EkstromSchmeling,JialunSahlsten1,JialunSahlsten2,AHW3} and the books \cite{Mattila2,Sahlsten} for broader discussion of these connections.

The decay (or non-decay) of the Fourier transform of $\mu$, and at what \textit{rate} (e.g. \textit{polylogarithmic} 
$$|\widehat{\mu}(\xi)| = O((\log |\xi|)^{-\beta})$$ 
for some $\beta > 0$ or \textit{polynomial} 
$$|\widehat{\mu}(\xi)| = O(|\xi|^{-\alpha})$$ 
for some $\alpha > 0$ as $|\xi| \to \infty$), can be considered a kind of measure of \textit{smoothness} of the measure $\mu$.  In general, if $\mu$ is any Radon measure on $\R^d$ that is absolutely continuous with respect to the $d$-dimensional Lebesgue measure, then by the Riemann-Lebesgue lemma, we always have $\widehat{\mu}(\xi) \to 0$ as $|\xi| \to \infty$. On the other hand, if $\mu$ gives positive mass to a proper non-empty hyperplane $V$ in $\R^d$, then it is an exercise to check that $\widehat{\mu}(\xi) \not\to 0$ as $|\xi| \to \infty$ along $\xi \in V^\perp$. In $\R$ this means that if $\mu$ contains atoms in its support, then the Fourier transform of $\mu$ cannot decay at all. On the other hand, if $\mu$ is \textit{curved} in some sense, e.g. in the case of the $d-1$ dimensional area measure $\sigma_{d-1}$ the sphere $\mathbb{S}^{d-1}$, Van der Corput lemma \cite{Mattila2} implies that $|\widehat{\sigma}_{d-1}(\xi)| = O(|\xi|^{-(d-1)/2})$. Similar results also work for cones by Fraser \textit{et al.} \cite{FHK}, more general submanifolds of non-vanishing Gaussian curvature (see e.g. \cite{Mattila2}), which also manifest in \textit{stationary phase approximations} in semiclassical analysis, see e.g. Zworski's book \cite{Zworski} on their use in semiclassical analysis. However, most attractors $K_\Phi$ to IFSs are instead broken and fractal-like sets, so it is unclear what is the analogue of curvature in these cases, which provides challenges in estimating $\widehat{\mu}$.

In fractal cases new issues can arise due to arithmetic reasons. For example, in $\R$, if $\mu$ is any self-similar measure for the IFS 
$\Phi_1 =  \Big\{x \mapsto \frac{x}{3},x \mapsto \frac{x}{3}+\frac{2}{3} \Big\}$
defining the middle $1/3$ Cantor set, then it is not difficult to see that $\widehat{\mu}(3^n) = \mu(1) \neq 0$ for all $n \in \N$, so in particular $\widehat{\mu}(\xi)$ does \textit{not decay} as $|\xi| \to \infty$. On the other hand, if $\mu$ is any non-atomic self-similar measure for the IFS 
$\Phi_2 =  \Big\{x \mapsto \frac{4}{5} x-1, x \mapsto \frac{4}{5} x +1 \Big\},$
then $\widehat{\mu}(\xi)\to 0$ with at least a \textit{polylogarithmic rate} as $|\xi| \to \infty$ by a work of Dai \cite{Dai} from 2012. Moreover, if $\mu$ is any non-atomic self-similar measure for the IFS 
$\Phi_3 =  \Big\{x \mapsto \frac{x}{2},x \mapsto \frac{x}{3}+\frac{2}{3} \Big\},$
then $\widehat{\mu}(\xi)\to 0$ with at least a polylogarithmic rate as $|\xi| \to \infty$ by a work of Li and the author \cite{JialunSahlsten1} from 2019. Furthermore, if $\mu$ is any non-atomic self-conformal measure 
$\Phi_4 =  \Big\{x \mapsto \frac{1}{x+1},x \mapsto \frac{1}{x+2} \Big\}$
defining the badly approximable numbers with continued fraction expansion bounded by $2$, then $\widehat{\mu}(\xi) \to 0$ with a \textit{polynomial rate} as $|\xi| \to \infty$ by a work of Jordan and the author \cite{JordanSahlsten} from 2016. Here the key differences between the IFSs is that the self-similar measures $\mu$ for $\Phi_1$ are invariant by multiplication by $3$ while for $\Phi_2$ there is no longer invariance under multiplication by $5/4$ and also the sequence $\dist((5/4)^n,\Z)$ is not square summable. In the case of $\Phi_3$ the inverse contractions $2$ and $3$ are different prime numbers, and $\Phi_4$ is non-linear. See also Figure \ref{fig:1} below for numerics on the Fourier transforms of stationary measures associated to couple of other IFSs.

\begin{figure}[ht!]
\hspace{-0.25in} \includegraphics[scale=0.315]{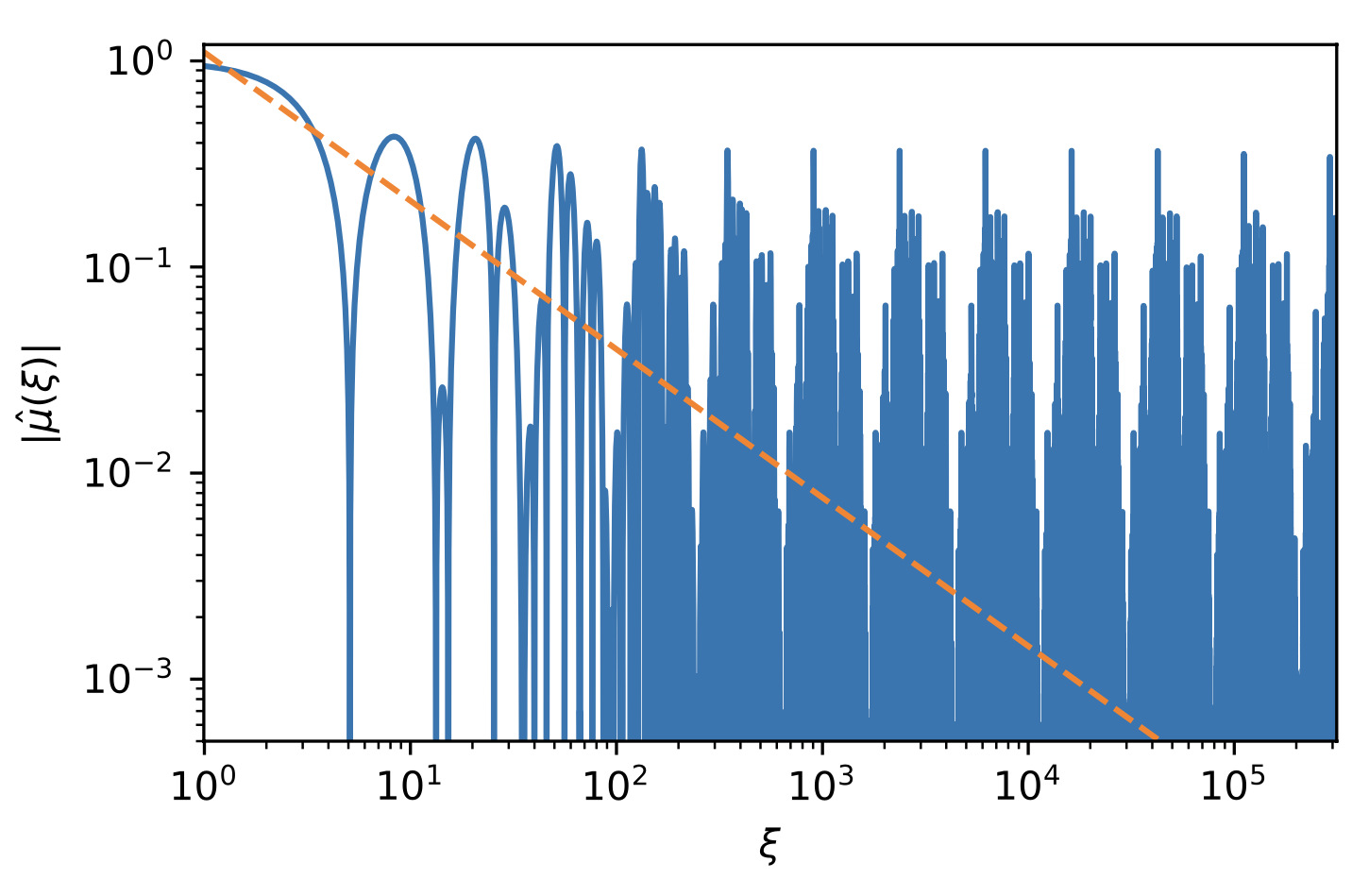}\qquad \includegraphics[scale=0.207]{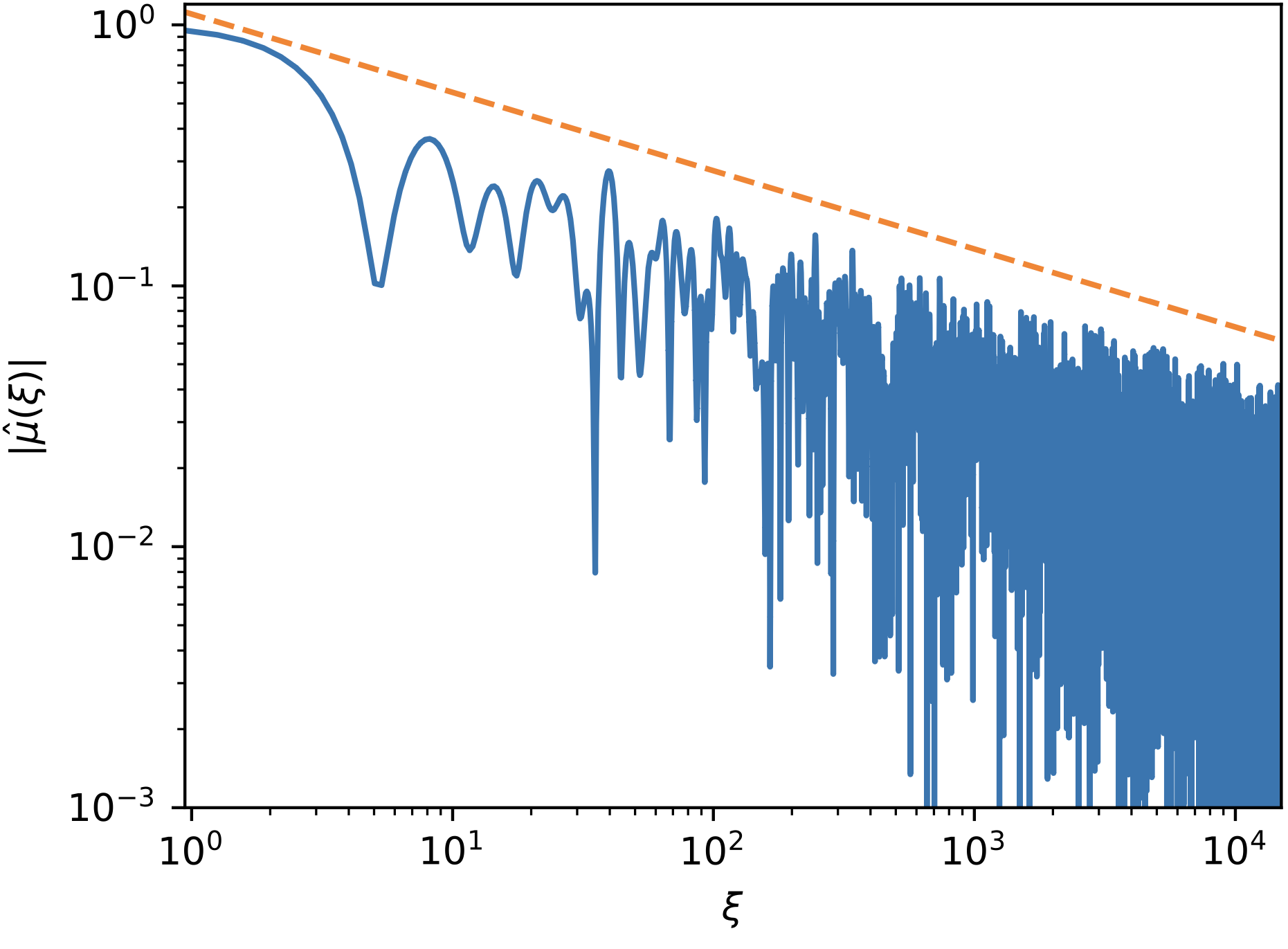}
\caption{\textit{Left}: Graph of $|\widehat{\mu}(\xi)|$ for a self-similar measure (Bernoulli convolution) associated to the IFS 
$\{x \mapsto \lambda x - 1,x \mapsto \lambda x + 1\},$
where $\lambda = (\frac{\sqrt{5}+1}{2})^{-2}$. Here $\lambda^{-2}$ is a Pisot number, so there is no Fourier decay for $\mu$, see Theorem \ref{thm:erdossalem} below. \textit{Right}: Graph of $|\widehat{\mu}(\xi)|$ for a self-conformal measure associated to the non-linear IFS 
$\{x \mapsto \tfrac{x}{3} + \tfrac{x^{2}}{6},x \mapsto \tfrac{x}{3} + \tfrac{2}{3}\}.$
The non-linearity of the IFS guarantees that there is polynomial Fourier decay for $\mu$, see Theorem \ref{thm:StevensSahlsten} below. The orange dotted line in both pictures is the polynomial function $|\xi|^{-\alpha/2}$ for reference, where $\alpha = 0.6 \approx \Hd \mu$, the Hausdorff dimension of $\mu$ \cite{Falconer}, in both cases. This suggests the second measure is potentially close to being a \textit{Salem measure}: $|\widehat{\mu}(\xi)| \lesssim |\xi|^{-\frac{\Hd \mu}{2}}$ as $|\xi| \to \infty$. Pictures by C. Wormell \cite{W1,W1a,W2}.}
\label{fig:1}
\end{figure}

These examples highlight the potential obstructions to the Fourier decay and its rate. In the self-similar case they are more of algebraic in nature, but in the self-conformal case the non-linearity in the IFS will be reflected in the Fourier transform. We remark that if instead of studying pseudorandom measures $\mu$ like the ones before, one \textit{randomises} measure $\mu$ some way, one can overcome some of these issues such as for random Cantor sets or for $\mu$ arising from the Brownian motion. The reason for these is that the added randomness will break some of the obstructions (like the symmetries arising from algebraic assumptions), see e.g. various works starting from Kahane articles \cite{KahaneLevel,KahaneImage} from 1980s and further works in the topic such as \cite{FOS,FS,ShmerkinSuomala,FalconerJin}.  In this manuscript we will only focus on \textit{deterministic} fractals, and see how the \textit{pseudo-randomness} due to algebraic or non-linearity reasons of the IFS $\Phi$ influences the Fourier transforms of stationary measures $\mu$ for IFSs. \newline

\textbf{Topics discussed in the paper.} First of all, in Section \ref{sec:selfsim}, we review results on Fourier transforms of self-similar measures in $\R$ such as Bernoulli convolutions and outline the recently developed methods used in the case where the IFS is not equicontractive. In Section \ref{sec:selfcon} we study the case of self-conformal measures in $\R$ associated to general $C^{1+\alpha}$ IFSs that are not self-similar and various methods based on thermodynamical formalism, additive combinatorics and renewal theory are introduced here. Next in Section \ref{sec:higherdim} we go to higher dimensional systems including self-similar measures that includes a proof of subpolynomial decay of self-similar measures using spectral information of transfer operators for random walks on rotation groups that has not yet previously been published (see Theorem \ref{thm:FourierSelfSimRd}) based on the work of Lindenstrauss and Varj\'u \cite{LV}. Here we will also outline the state-of-the-art and challenges self-affine and general $C^{1+\alpha}$ IFSs in connection to Fractal Uncertainty Principles in quantum chaos. Finally in Section \ref{sec:prospects} we discuss potential future directions and prospects to this topic.

\section{Self-similar iterated function systems in $\R$}\label{sec:selfsim}

In this part we will go through the simplest setting of \textit{self-similar} iterated function systems in $\R$, that is, systems where the contractions are of the form
$$f_j(x) = r_j x+ b_j, \quad x \in \R,$$
for some $0 < r_j < 1$ and $b_j \in \R$, $j \in \cA$. We will first consider the \textit{homogeneous case}, where $r_j \equiv \lambda$ for all $j \in \cA$ for some fixed $0 < \lambda < 1$, which has also been historically the most well established, and then move to the general inhomogeneous case.

\subsection{Equicontractive systems and Bernoulli convolutions}

Consider $0 < \lambda < 1$ and $\cA = \{1,2\}$ and the iterated function system $\Phi_\lambda := \{f_1,f_2\}$  determined by two maps $f_1(x) = \lambda x - 1$ and $f_2(x) = \lambda x + 1$. Then we can associate a non-empty compact $C_\lambda \subset \R$ satisfying $f_1(C_\lambda) \cup f_2(C_\lambda) = C_\lambda$ and a measure $\mu_\lambda$ on $\R$, called the \textit{Bernoulli convolution} satisfying $\frac{1}{2} f_1 \mu_\lambda + \frac{1}{2} f_2\mu_\lambda = \mu_\lambda$, where push-forward $f\mu$ is defined as $f\mu(A) := \mu(f^{-1}A)$ for any Borel $A \subset \R$. In other words, $\mu_\lambda$ is the distribution of the random variable $\sum \pm \lambda^{k}$ with $\pm$ chosen i.i.d. with equal probability. If $\lambda < 1/2$, then $C_\lambda$ is a self-similar Cantor set with the middle $1-2\lambda$ interval removed. When $\lambda \geq 1/2$, then $C_\lambda$ is an interval and the IFS has potentially complicated overlaps. 

Due to the homogeneity of the iterated function system $\Phi_\lambda = \{f_1,f_2\}$, that is, all the contraction ratios of $f_1$ and $f_2$ are equal to the same number $\lambda$ the measure $\mu_\lambda$ can be represented as an infinite convolution of Bernoulli measures $\frac{1}{2} \delta_{\lambda^{n}} + \frac{1}{2} \delta_{-\lambda^{n}}$ on $\R$ that allows us to represent the Fourier transform of $\mu_\lambda$ conveniently as
$$\widehat{\mu}_\lambda(\xi) = \prod_{n = 0}^\infty \cos(2\pi \lambda^n \xi), \quad \xi \in \R.$$
This expression reveals us how potentially the \textit{algebraic} properties of the number $\lambda$ to the asymptotic behaviour of $\widehat{\mu}_\lambda$. In particular, using this expression one can find out that the distances of $\lambda^{-n}$ to the integer lattice will play a key role. In particular, we say that $\lambda^{-1}$ is a \textit{Pisot number} if it is a real algebraic integer whose Galois conjugates are all strictly less than $1$ in absolute value, which is equivalent to the fact that the sequence $\dist(\lambda^{-n},\Z)$, $n \in \N$, is square summable. This allowed Erd\H{o}s and Salem in the late 1930s and early 1940s to prove the following characterisation:

\begin{theorem}\label{thm:erdossalem}
$\widehat{\mu}_\lambda(\xi) \to 0$ as $|\xi| \to \infty$ if and only if $\lambda^{-1}$ is not a Pisot number or $\lambda = 1/2$.
\end{theorem}

The direction that the decay of $\widehat{\mu}_\lambda$ implies $\lambda^{-1}$ is not Pisot was proved by Erd\H{o}s in 1939 \cite{erdos2} and the opposite was done by Salem in 1944 \cite{salem3}. 

Now, what about the \textit{rate} of decay for $\widehat{\mu}_\lambda$? This is a particularly attractive question because of its relation to studying the \textit{absolute continuity} of the Bernoulli convolution. A conjecture here is that in fact $\lambda^{-1}$ when $\lambda > 1/2$ not being Pisot number should characterise the absolute continuity of $\mu_\lambda$, which has been stated e.g. in Varj\'u's survey \cite{Varju}. For example, in the argument of Shmerkin \cite{ShmerkinConv} from 2012, having polynomial decay rate for $\widehat{\mu}_{\lambda^k}(\xi)$ as $|\xi| \to \infty$ for some $k \in \N$ is crucial ingredient in the method of proving absolute continuity of $\mu_\lambda$. However, there are only very few examples of Bernoulli convolutions that have polynomial Fourier decay and in some cases it may not even be possible to have such rapid decay. For example, if $\lambda^{-1}$ is a \textit{Salem number}, i.e. a real algebraic integer whose Galois conjugates are all less than one and at least one has absolute value exactly $1$. Then, by a work of Peres, Schlag and Solomyak \cite[Lemma 5.2]{PSS} from 2000 we know that

\begin{theorem}\label{thm:salem}
If $\lambda^{-1}$ is a Salem number, then $|\widehat{\mu}_\lambda|$ cannot decay faster than subpolynomial rate.
\end{theorem}

Theorem \ref{thm:salem} imposes some algebraic obstructions for the rate of Fourier decay, but actually finding explicit examples where the rate for $\widehat{\mu}_\lambda$ can be estimated even suboptimally is very challenging. The following is known in the rational case, see the works of Kershner \cite{Kershner} from 1936 and Dai \cite{Dai} from 2012:
\begin{theorem}
If $\lambda = \frac{p}{q} \in \Q$ with $q > p > 1$ and $p,q$ are relatively prime, then 
$$|\widehat{\mu}_\lambda(\xi)| = O((\log|\xi|)^{-\gamma}), \quad |\xi| \to \infty,$$ 
where $\gamma = -\log \cos (\frac{\pi}{2q})/\log (2\frac{\log q}{\log p}) > 0$.
\end{theorem}

Note here it is crucial that $p > 1$ since for $\lambda = \frac{1}{q}$ the measure $\mu_\lambda$ becomes $\times q$ invariant, so it cannot be Rajchman by the same argument as with $\times 3$ invariant measures. Moreover, having just the existence of one Galois conjugate outside the unit circle is enough for the polylogarithmic decay, as shown by Bufetov and Solomyak in \cite{BufSol} from 2014:

\begin{theorem}
If $\lambda^{-1}$ is a real algebraic number with at least one Galois conjugate of norm strictly bigger than $1$, then for some $\gamma > 0$ we have 
$$|\widehat{\mu}_\lambda(\xi)| = O((\log|\xi|)^{-\gamma}),$$ 
as $|\xi| \to \infty.$
\end{theorem}

To get any faster than polylogarithmic decay, one can push this idea by assuming \textit{all} the Galois conjugates are outside the unit circle and the product of the conjugates is
$\pm 2$. Such algebraic integers are called \textit{Garsia numbers}. For these cases, Dai, Feng and Wang \cite{DFW} in 2007 managed to establish \textit{polynomial} Fourier decay:

\begin{theorem}
If $\lambda^{-1}$ is a Garsia number, then $\widehat{\mu}_\lambda(\xi) \to 0$ at a polynomial rate as $|\xi| \to \infty$.
\end{theorem}

In fact, \cite{DFW} generalises to equally weighted self-similar measures $\mu$ in the wider class of IFSs $\Phi = \{x \mapsto \lambda x + b_j : j \in \cA\}$. Recently in 2023 \cite{Streck}, Streck gave further examples of measures $\mu_\lambda$ with polynomial Fourier decay for arbitrary weighted $\mu$ for IFSs $\Phi = \{x \mapsto \lambda x + b_j : j \in \cA\}$ where all $b_j \in \Q$. The approach in \cite{Streck} is to study a self-affine measure on $\R^d$ projecting onto $\mu_\lambda$.

Given the positive results on the rate of decay of $\widehat{\mu}_\lambda$, what kind of algebraic features of $\lambda^{-1}$ \textit{characterise} the rate of Fourier decay? As we noted in Theorem \ref{thm:salem}, there are obstructions, but it is not clear what is the rate. P. Varj\'u (in an email communication with the author in 2023) suggested the following problem towards this:

\begin{problem}\label{prob:salemdecay}
If $\lambda^{-1}$ is Salem, does $|\widehat{\mu}_\lambda(\xi)| \to 0$ as $|\xi|\to \infty$ with polylogarithmic rate?
\end{problem}

For non-Salem $\lambda^{-1}$, it is not clear at all if $\mu_\lambda$ has very rapid Fourier decay or not. Thus the following problem that arises in this setting is likely very difficult:

\begin{problem}\label{prob:powerdecay}
Characterise non-Salem algebraic numbers $\lambda^{-1}$ such that $\widehat{\mu}_\lambda$ has faster than polylogarithmic Fourier decay.
\end{problem}

This would require one to use in a more fundamental way the \textit{translations} of the IFS in combination with the powers of $\lambda$ when computing the Fourier transform. In the related study of absolute continuity of $\mu_\lambda$, motivated by the advances of dimension theory \cite{Hochman,Shmerkin}, one could wonder if the main obstructions to rapid Fourier decay come from the existence of \textit{exact-} or \textit{superexponential overlaps} in the IFS $\Phi_\lambda$ produced in-part by the translations? However, we note here the work of Simon and V\'ag\'o \cite{SimonVago} from 2019, who provided new obstructions to absolute continuity that do not arise from exact overlaps in $\Phi_\lambda$, in particular, these conditions could also influence the possible polynomial rate of decay of Fourier transform of $\mu_\lambda$. For more details on this field, see the recent surveys by Varj\'u \cite{Varju,Varju2}. Thus, towards Problem \ref{prob:powerdecay}, it would be interesting to try to upgrade the results on the polylogarithmic decay by developing new methods that use more deeply the algebraic properties of the number $\lambda$. 

Also, going beyond the algebraic case, i.e. studying \textit{transcendental} $\lambda^{-1}$, not much is known. There has been much progress in the dimension theory of $\mu_\lambda$ recently for transcendental $\lambda^{-1}$ by Varj\'u \cite{VarjuT} in 2019, which motivated that something could be done for also the Fourier transform of $\mu_\lambda$. P. Varj\'u suggested that given that Pisot or Salem numbers provide obstruction to decay or rapid decay, perhaps if $\lambda^{-1}$ is quantitatively very far from them, one could have good decay. A concrete problem on this direction could be:

\begin{problem}\label{prob:transcdecay}
If $\lambda^{-1}$ is transcendental and not very well approximable by Salem numbers, then does $\widehat{\mu}_\lambda$ decay polynomially?
\end{problem}

Now, given that it has been hard to actually prove explicit examples of $\lambda$ such that $\mu_\lambda$ exhibits polynomial Fourier decay, it has been useful to \textit{randomise} $\lambda$ and try to prove that for ``most'' $\lambda$ there is polynomial Fourier decay. Erd\H{o}s proved in \cite{Erdos} from 1940 that for \textit{Lebesgue almost every} $\lambda \in (0,1)$ the Fourier transform of the Bernoulli convolution $\mu_\lambda$ decays polynomially. Furthermore, Kahane \cite{Kahane} in 1969 upgraded this to saying that the exceptional set of $\lambda \in (0,1)$ for which $\widehat{\mu}_\lambda$ does not have polynomial Fourier decay has Hausdorff dimension $0$. The method now is known as the popular and useful \textit{Erd\H{o}s-Kahane method}, which has many generalisations. In fact, this result works for any homogeneous IFS without common fixed points, the following is due to Shmerkin \cite{ShmerkinConv} from 2012:

\begin{theorem}\label{thm:shmerkin}
There exists $E \subset (0,1)$ with $\Hd E = 0$ such that if $\lambda \in (0,1) \setminus E$ and $b_j \in \R$, $j \in \cA$, for finite $\cA$ are distinct, then any non-atomic self-similar measure $\mu$ associated to the IFS $\{x \mapsto \lambda x + b_j : j \in \cA\}$ satisfies 
$$|\widehat{\mu}(\xi)| = O(|\xi|^{-\alpha})$$ 
as $|\xi| \to \infty$ for some $\alpha > 0$.
\end{theorem}

Beyond Theorem \ref{thm:shmerkin} not much is known, and the general cases such as Problems \ref{prob:powerdecay} and \ref{prob:transcdecay} remain as important challenge. For general Bernoulli convolutions and self-similar measures, one can say something about the \textit{average} decay of Fourier transforms for $0 < \lambda < 1/2$ when $K_\Phi$ is a Cantor set. In this case Strichartz \cite{S1} in 1990 proved:

\begin{theorem}\label{thm:Stri}
If $0 < \lambda < 1/2$, then
$$\frac{1}{2R}\int_{-R}^R |\widehat{\mu}_\lambda(\xi)|^2 \, d\xi = O(R^{-\frac{\log 2}{\log(1/\lambda)}}), \quad R \to \infty.$$
\end{theorem}

We also highlight the work of Fan and Lau \cite{FanLau2}, where multiperiodic functions of the form $G(x) = \prod_{n=1}^{\infty} g(x/2^n)$ were considered which represent the Fourier transforms of self-similar functions or measures such as Bernoulli convolutions. Strichartz's work was then refined by Tsujii \cite{Tsujii} in 2015 who also included the overlapping case, where a kind of large deviation principle holds for the Fourier transform $\widehat{\mu}_\lambda(\xi)$:
\begin{theorem}\label{thm:Tsujii}
For any $0 < \lambda < 1$, we have
$$\lim_{\alpha \searrow 0} \limsup_{t \to \infty} \frac{1}{t} \log |\{\xi \in [-e^{t},e^t] : |\widehat{\mu}_\lambda(\xi)| \geq e^{-\alpha t}\}| = 0.$$
\end{theorem}
This means that $|\widehat{\mu}_\lambda(\xi)|$ decays polynomially except along a very sparse set of frequencies. This was also recently generalised by Khalil \cite{Khalil} in 2023, which we discuss more in the Section \ref{sec:higherdim} on higher dimensions. The works of Strichartz, Kaufman and Tsujii also work for more general IFSs \cite{S2,S3}, which are not necessarily equicontractive (in particular Tsujii's work holds for \textit{any} non-trivial self-similar measure) that is, when the contractions $r_i \neq r_j$ for some pair $i \neq j$. This has only more recently become quite active, which we will now discuss.

\subsection{General self-similar systems}

In the case of a general self-similar IFSs $\Phi = \{x \mapsto r_j x + b_j : j \in \cA\}$, where some $r_i \neq r_j$, the lack of convolution structure makes it more difficult to get similar algebraic results. The main results for general IFS involves the generalisation of Theorem \ref{thm:Stri} to IFSs such that $\log r_i / \log r_j$ are all rational numbers, and Tsujii's result (Theorem \ref{thm:Tsujii}) to any general self-similar IFSs. However, the works of Strichartz and Tsujii give no examples of non-equicontractive self-similar measures that are Rajchman.

A quite different case compared to Bernoulli convolutions occurs when two contractions $r_j$ in the IFS are \textit{not} powers of each other, that is, $\frac{\log r_i}{\log r_j} \in \R \setminus \Q$. In this case, the additive subgroup generated by the log-contractions $-\log r_j$ is dense in $\R_+$, where as for Bernoulli convolutions log-contractions just lie in the lattice $(-\log \lambda)\Z$. A toy model to consider here would be the IFS 
$$\Phi = \{x \mapsto f_1(x) = x/2,x \mapsto f_2(x) = x/3 + 2/3\}$$ 
where the lengths of the intervals involve powers of $2$ and $3$, which are prime and so $\frac{\log 2}{\log 3} \in \R \setminus \Q$. Consider a self-similar measure $\mu$ associated to $\Phi$ with equal weights, that is, we have 
$$\mu = \frac{1}{2}f_1(\mu) + \frac{1}{2}f_2(\mu).$$ 
By iterating this, Cauchy-Schwarz allows us to reduce the problem of bounding $|\widehat{\mu}(\xi)|^2 = \widehat{\mu}(\xi)\overline{\widehat{\mu}(\xi)}$ to exponential sums involving products involving powers of $2$ and $3$:
\begin{align} \label{eq:FourierBound} |\widehat{\mu}(\xi)|^2 \leq \iint \frac{1}{2^n} \sum_{ \a \in \{1,2\}^n } e^{-2\pi i \xi 2^{-k(\a)} 3^{-n+k(\a)}  (x-y)} \, d\mu(x) \, d\mu(y),
\end{align}
where $k(\a) := \sharp \{j = 1,\dots, n : a_j = 1\}$. Since $\frac{\log 2}{ \log 3} \in \R \setminus \Q$, one would then expect that when we pick the word $\a \in \cA^n$ randomly with weights $2^{-n}$ and fix $x-y$, the numbers 
$$2^{-k(\a)} 3^{-n+k(\a)} (x-y)$$ 
should exhibit some form of ``non-concentration'' (i.e. not concentrate near a point), which should translate to cancellations in the exponential sums in \eqref{eq:FourierBound} as long as $x-y$ is not too small. However, as we can represent 
$$2^{-k(\a)} 3^{-n+k(\a)} (x-y) = (3/2)^{k(\a)} 3^{-n}(x-y),$$ 
it is not clear how the single parameter orbits $(3/2)^{k} z_n$ distributes mod $1$ under multiplication by $3/2$ when $k = k(\a)$ grows in $n$ and $z_n = 3^{-n}(x-y)$.

A way around this problem is to use a \textit{stopping time argument}, where we try to exploit the \textit{distribution} of the fluctuations of $2^{-k(\a)} 3^{-n+k(\a)}$ nearby a given value $e^{-t}$, and use irrationality of $\frac{\log 2}{\log 3}$ so that this distribution becomes absolutely continuous as $t \to \infty$. This is roughly speaking the renewal theoretic approach to studying Fourier transforms of self-similar measures, which is done in more generality by Li and the author in \cite{JialunSahlsten1} in 2019. The idea was inspired by the renewal theoretic approach to study the Fourier transforms of stationary measures for random walks on groups by Li \cite{Li} in 2018, which we will discuss more in detail in the self-conformal section as the measures it applied to come from non-linear IFSs. See also the work of Lalley \cite{Lalley} from 1989. The renewal theoretic idea that allows us to study Fourier transforms of self-similar measures for IFSs $\Phi = \{x \mapsto r_j x + b_j : j \in \cA\}$ when $\log r_i / \log r_j$ is irrational for some $i \neq j$:

\begin{theorem}\label{thm:lisahlsten}
Let $\Phi = \{x \mapsto r_j x + b_j : j \in \cA\}$ such that no two pair of the maps $x \mapsto r_j x + b_j$ has a common fixed point. Then every self-similar measure associated to $\Phi$ and a probability vector $p_j \in (0,1)$, $j \in \cA$, is Rajchman if there exists $i \neq j$ such that $\frac{\log r_j}{\log r_j} \in \R \setminus \Q$.
\end{theorem}

\begin{proof}[Sketch of the proof of Theorem \ref{thm:lisahlsten}]
Let us consider still the IFS 
$$\Phi = \{x \mapsto f_1(x) = x/2,x \mapsto f_2(x) = x/3 + 2/3\}.$$ 
Write $\xi = s e^{t}$ for some $s \in \R$ and $t > 0$ and consider the words $\cW_t$ of potentially \textit{different }lengths where for all $\a \in \cW_t$ we have 
$$2^{-k(\a)} 3^{-n+k(\a)} \approx e^{-t}.$$ 
Thus we replace in \eqref{eq:FourierBound} the average over $\{1,2\}^n$ by a weighted sum over $\cW_t$:
\begin{align} \label{eq:FourierBoundStop} |\widehat{\mu}(\xi)|^2 \lesssim \iint\limits_{|x-y| > \delta} \sum_{ \a \in \cW_t } 2^{-|\a|} e^{-2\pi i \xi 2^{-k(\a)} 3^{-n+k(\a)}  (x-y)} \, d\mu(x) \, d\mu(y) + \delta^{\eps_1},
\end{align}
which is possible because the words $\cW_t$ form a partition and splitting of the integration was done with some $\delta = \delta(\xi) \to 0$ as $|\xi| \to \infty$. The term $\delta^{\eps_1}$ comes from the fact that $\mu$ is \textit{non-atomic} self-similar measure, so for all $\delta > 0$, we have 
$$\mu \times \mu((x,y) : |x-y| \leq \delta) = \int \mu(B(x,\delta)) \, d\mu(x) \lesssim \delta^{\eps_1}$$ 
for some $\eps_1 > 0$ \cite{FL}. Taking logarithms, this translates the study of the large $|\xi|$ limit in \eqref{eq:FourierBoundStop} to understanding the limit of the distributions 
$$\sum_{k = 1}^{n_{t}} X_k-t$$ 
as $t \to \infty$, where $X_k = \log 2$ with probability $1/2$ and $X_k = \log 3$ with probability $1/2$. Here $n_t := n_t(w)$, $w \in \{1,2\}^\N$, is the \textit{stopping time}, that is, the smallest $n$ such that 
$$2^{-k(w_1 \dots w_n)} 3^{-n+k(w_1\dots w_n)} < e^{-t}.$$ 
In other words, setting 
$$g_{s,x,y}(z) := e^{-2\pi i s  (x-y) e^{-z}}, \quad z \in \R,$$ we see in \eqref{eq:FourierBoundStop}, once $x,y$ are fixed, that
$$\sum_{ \a \in \cW_t } 2^{-|\a|} e^{-2\pi i \xi 2^{-k(\a)} 3^{-n+k(\a)}  (x-y)} = \E_t g_{s,x,y}\Big(\sum_{k = 1}^{n_{t}} X_k-t\Big),$$
where $\E_t$ is the expectation associated to the probability measure $\sum_{ \a \in \cW_t } 2^{-|\a|}\delta_\a$ on $\cW_t$. 

The limiting distribution of $\sum_{k = 1}^{n_{t}} X_k-t$ as $t \to \infty$ can be found using \textit{renewal theory}. In particular, the irrationality of $\frac{\log 2}{ \log 3}$ and the \textit{renewal theorem} for stopping times (e.g. Kesten's renewal theorem \cite{Kesten1974} from 1974) guarantees the limit distribution of $\sum_{k = 1}^{n_{t}} X_k-t$ as $t \to \infty$ is some absolutely continuous measure $\lambda$, so
$$\lim_{t \to \infty} \E_t g_{s,x,y}\Big(\sum_{k = 1}^{n_{t}} X_k-t\Big) = \int g_{s,x,y}(z) \, d\lambda(z).$$
Now, by the definition of $g_{s,x,y}$, the oscillating integral $\int g_{s,x,y}(z) \, d\lambda(z)$ is actually the Fourier transform of push-forward of $\lambda$ under $e^{-z}$ at frequency $s(x-y)$, so if $|s(x-y)|$ is arbitrarily large, Riemann-Lebesgue lemma guarantees 
$$\int g_{s,x,y}(z) \, d\lambda(z) \to 0.$$ 
This leads to $\widehat{\mu}(\xi) \to 0$ as $|\xi| \to \infty$ since $\xi = s e^t$. Now to make this argument more formal, one would need to choose $s$, $t$ and $\delta$ depending on $\xi$ in a way we can take a double limit in this procedure.
\end{proof}

Theorem \ref{thm:lisahlsten} leaves open the case what happens when $\log r_i / \log r_j \in \Q$ for some $r_i \neq r_j$. Could one have Rajchman property here for some cases? Compared to the result of Erd\H{o}s and Salem (Theorem \ref{thm:erdossalem}) that characterise the Rajchman property of Bernoulli convolutions $\mu_\lambda$ using the non-Pisot property of $\lambda^{-1}$, one could expect some cases arising from non-Pisot numbers to still give Rajchman property in the inhomogeneous case. This was indeed confirmed and fully classified by Br\'emont \cite{Bremont} in 2021:

\begin{theorem}\label{thm:bremont}
Let $\Phi = \{x \mapsto r_j x + b_j : j \in \cA\}$ such that the maps $x \mapsto r_j x + b_j$ do not have common fixed points. If every self-similar measure $\mu$ associated to $\Phi$ and a probability vector $p_j \in (0,1)$, $j \in \cA$, is not Rajchman, then there exists $0 < \lambda < 1$ such that $\lambda^{-1}$ is a Pisot number and $\ell_j \in \Z_+$, $j \in \cA$, with $\mathrm{gcd}(\ell_j : j \in\cA) =1$ such that $r_j = \lambda^{\ell_j}$ for all $j \in \cA$ and $\Phi$ is conjugated by a similarity to an IFS such that $b_j \in \Q(\lambda)$ for all $j \in \cA$.
\end{theorem}

Br\'emont's work \cite{Bremont} also classifies the structure of the probability vectors $p_j \in (0,1)$, $j \in \cA$, which give rise to non-Rajchman self-similar measures $\mu$ when $r_j = \lambda^{\ell_j}$ for all $j \in \cA$ with some Pisot $\lambda^{-1}$ and $\Phi$ is conjugated by a similarity to an IFS such that $b_j \in \Q(\lambda)$ for all $j \in \cA$.

Both Theorems \ref{thm:lisahlsten} and \ref{thm:bremont} raise the question of what kind of \textit{rates} of decay for the Fourier transform $\widehat{\mu}$ we could expect in the Rajchman case. In the renewal theoretic approach the main parts that influence the Fourier decay rate are the Frostman property of $\mu$, the rate at which the renewal operators 
$$\E_t g_{s,x,y}\Big(\sum_{k = 1}^{n_{t}} X_k-t\Big)$$ 
converge to the expectations with respect to an absolutely continuous distribution $\lambda$ as $t \to \infty$ and the rate of decay of the Fourier transform of the diffeomorphic push-forward of $\lambda$. Here the key issue is that the rate in the renewal theorem for the random walks $X_n$, considered here in $\R$, cannot be any faster than polynomial in $t$. As $\xi = se^{t}$, this translates to at most polylogarithmic decay for the Fourier transform. Indeed, in the work by Li and the author \cite{JialunSahlsten1} from 2019, it was also shown that when for some pair $r_i$ and $r_j$ the ratio $\log r_j / \log r_j$ is not very well approximable by rational numbers (e.g. it is not a \textit{Liouville number}), then one can obtain a \textit{polylogarithmic rate} of Fourier decay:

\begin{theorem}\label{thm:lisahlstenRate}
Let $\Phi = \{x \mapsto r_j x + b_j : j \in \cA\}$ such that none of the maps $x \mapsto r_j x + b_j$ have common fixed points and that there exists $i \neq j$ such that $\frac{\log r_j}{\log r_j}$ is not a Liouville number. Then for every self-similar measure associated to $\Phi$ and a probability vector $p_j \in (0,1)$, $j \in \cA$, there exists for some $\gamma > 0$ we have 
$$|\widehat{\mu}(\xi)| = O((\log|\xi|)^{-\gamma})$$
as $|\xi| \to \infty.$
\end{theorem}

Theorem \ref{thm:lisahlstenRate} was later generalised to a more general Diophantine condition on the IFS by Algom, Hertz and Wang \cite{AHW} in 2021, which we discuss more in the Section \ref{sec:selfcon} on self-conformal sets. Theorem \ref{thm:lisahlstenRate} leaves the case open where $r_j = \lambda^{\ell_j}$ for all $j \in \cA$. If $\lambda^{-1}$ is an algebraic integer, then potentially the locations of the Galois conjugates, as with the Bernoulli convolution case, should play some role in controlling the rate of the Fourier transform of the self-similar measure. Indeed, in the work \cite{VarjuYu} Varj\'u and Yu in 2020 combined with a work of Bufetov and Solomyak \cite{BufSol} established

\begin{theorem}\label{thm:varjuyu}
Let $\Phi = \{x \mapsto r_j x + b_j : j \in \cA\}$ such that no two pair of the maps $x \mapsto r_j x + b_j$ has a common fixed point. Let $\lambda\in(0,1)$ such that $\lambda^{-1}$ is a real algebraic integer that is not a Pisot nor a Salem number and $\ell_j \in \Z_+$, $j \in \cA$, with $\mathrm{gcd}(\ell_j : j \in\cA) =1$ such that $r_j = \lambda^{\ell_j}$ for all $j \in \cA$. Then for every self-similar measure $\mu$ associated to $\Phi$ and a probability vector $p_j \in (0,1)$, $j \in \cA$, there exists for some $\gamma > 0$ such that 
$$|\widehat{\mu}(\xi)| = O((\log|\xi|)^{-\gamma})$$
as  $|\xi| \to \infty.$
\end{theorem}

Moreover, in \cite{VarjuYu} Varj\'u and Yu also completely characterised the self-similar sets that are sets of uniqueness to trigonometric series.

Why is it that we can only get \textit{polylogarithmic} rate of decay for the Fourier transform using the renewal theoretic approach? Fundamentally this is because the Cauchy-Schwarz inequality (e.g. \eqref{eq:FourierBoundStop}) destroys the whole contribution of the \textit{translations} $b_j$ from the IFS. This crucially prevents us to get anything faster than polylogarithmic decay using the renewal theoretic approach. Thus we need to use the translations of the IFS more crucial way. The following problem is plausible (e.g. given the polylogarithmic rate in Theorem \ref{thm:lisahlstenRate}):

\begin{problem}\label{prob:powerdecaynonhom}
Let $\Phi = \{x \mapsto r_j x + b_j : j \in \cA\}$ such that no two pair of the maps $x \mapsto r_j x + b_j$ has a common fixed point and there exists $i \neq j$ such that $\frac{\log r_j}{\log r_j}$ is not Liouville. Then every self-similar measure associated to $\Phi$ and a probability vector $p_j \in (0,1)$, $j \in \cA$, satisfies for some $\alpha > 0$ that 
$$|\widehat{\mu}(\xi)| = O(|\xi|^{-\alpha})$$
as $|\xi| \to \infty.$
\end{problem}

One can approach Problem \ref{prob:powerdecaynonhom} in a similar fashion as with the Bernoulli convolutions by \textit{randomising} the contractions $(r_j : j \in \cA)$ by adapting an Erd\H{o}s-Kahane type argument attaining a similar statement as Theorem \ref{thm:shmerkin}. This was done by Solomyak in \cite{Solomyak} in 2019:

\begin{theorem}\label{thm:solomyakPoly}
Let $\cA$ be finite. Then there exists $\cE \subset (0,1)^\cA$ of Hausdorff dimension $0$ such that the following holds: for any 
$$\mathbf{r}= (r_j : j\in \cA) \in (0,1)^\cA \setminus \cE \text{ and } \mathbf{b} = (b_j : j\in \cA) \in \R^\cA$$ such that 
$$\Phi_{\mathbf{r},\mathbf{b}} =  \{x \mapsto r_j x + b_j : j \in \cA\}$$ 
has no common fixed points, and any self-similar measure associated $\Phi_{\mathbf{r},\mathbf{b}}$ and a probability vector $p_j \in (0,1)$, $j \in \cA$, there exists $\alpha > 0$ such that 
$$|\widehat{\mu}(\xi)| = O(|\xi|^{-\alpha})$$
as $|\xi| \to \infty.$
\end{theorem}

No explicit example of a non-equicontractive self-similar measure in $\R$ is yet known that has polynomial Fourier decay. Given the polynomial Fourier decay exhibited by the Bernoulli convolution $\mu_\lambda$ for Garsia $\lambda^{-1}$, the following problem is plausible with potentially similar methods:

\begin{problem}\label{prob:Garsianonhom}
Let $\Phi = \{x \mapsto r_j x + b_j : j \in \cA\}$ such that the maps $x \mapsto r_j x + b_j$ do not have common fixed points. Let $0 < \lambda < 1$, $\lambda \neq 1/2$, be such that $\lambda^{-1}$ is a Garsia number and $\ell_j \in \Z_+$, $j \in \cA$, with $\mathrm{gcd}(\ell_j : j \in\cA) =1$ such that $r_j = \lambda^{\ell_j}$ for all $j \in \cA$. Construct probability vectors $p_j \in (0,1)$, $j \in \cA$, such that the self-similar measure $\mu$ for $\Phi$ associated to the probability vector satisfies for some $\alpha > 0$ that 
$$|\widehat{\mu}(\xi)| = O(|\xi|^{-\alpha})$$ 
as $|\xi| \to \infty$.
\end{problem}

However, going beyond this case, achieving polynomial Fourier decay in the self-similar case seems to be very challenging. On the other hand, this problem can be overcome by \textit{perturbing} the self-similar IFS by adding some \textit{non-linearity} to maps, where the non-linearity will amplify into more rapid Fourier decay rates. This will be the topic of the next section.

\section{Non-linear self-conformal iterated function systems in $\R$}\label{sec:selfcon}

In this section we will move to IFSs of the form $\Phi = \{f_j : j \in \cA\}$ with non-vanishing $f_j'$ such that $f_j'$ is a $\alpha$-H\"older regular function. We call these self-conformal $C^{1+\alpha}$ IFSs. These IFSs are very important also in the study of hyperbolic dynamics, as they arise from the Poincar\'e sections of Anosov flows and chaotic billiards (see e.g. \cite{Bowen,BowenSeries, AGY}). A famous example are the analytic IFSs given by the subshifts coming from the limit sets of Fuchsian groups describing hyperbolic trapped sets on non-compact hyperbolic surfaces \cite{BowenSeries}.

Most of the progress currently has been focused on $C^{\omega}$ IFSs (that is, real analytic $f_j$) and $C^2$ IFSs, but there has been also some recent progress in the $C^{1+\alpha}$ category, which we will discuss. If the maps $f_j \in \Phi$ are just similitudes, then the theory reduces to Section \ref{sec:selfsim}, so we are interested in the properly non-linear cases, that is, when for one $j \in \cA$ the derivative $f_j'$ is not a constant function. This might help in bounding the Fourier transforms of stationary measures for $\Phi$.

\subsection{Fourier decay via nonlinearity of the Gauss map}

We will start from the very concrete cases of non-linear IFSs that arise from Diophantine approximation and the \textit{Gauss map} 
$$x \mapsto G(x) := \frac{1}{x} \mod 1.$$ 
In the early 1980s Robert Kaufman wrote two articles \cite{Kaufman,Kaufman2} on constructing Rajchman measures on sets arising from Diophantine approximation, in particular the \textit{well-} and \textit{badly approximable numbers}. These are fractal sets of irrational numbers $x$ defined by the rate at which they can be approximated by rational numbers. The first construction was on the well-approximable numbers $W(\alpha)$ for rate exponent $\alpha > 0$, which Kaufman proved to be \textit{Salem sets}, that is, for any $\eps > 0$, they support measures $\mu$ such that $|\widehat{\mu}(\xi)| \lesssim |\xi|^{-\delta/2 + \eps}$ as $|\xi| \to \infty$, where $\delta = \Hd W(\alpha)$. See also the work of Bluhm \cite{Bluhm} from 1998 for more general settings and \cite{PoVe,PVZZ} from 2000 and 2022 for the importance of constructing Rajchman measures on fractals arising from Diophantine approximation.

It turns out that also in the case of \textit{badly approximable} numbers $B_N$, which \textit{is} an attractor of the IFS 
$$\Phi_N := \{x \mapsto f_a(x) = 1/(x+a) : a = 1,2,...,N\},$$ 
where the maps $f_a$ form the inverse branches of the Gauss map $G(x)$ up to the $N$th branch, one can also construct Rajchman measures with polynomial Fourier decay. However, now the structure of the set $B_N$ is very different from the limsup type set of well-approximable numbers $W(\alpha)$.  The nonlinearity in the system manifests in the distribution of the distortion factors $f_{\a}'' (x)/ f_\a(x)$, where $f_a = f_{a_1} \circ \dots \circ f_{a_n}$. Due to the specific Diophantine nature of the distortions coming from the Gauss map, the distortions exhibit a form of \textit{non-concentration}, which Kaufman exploited to construct a suitable Bernoulli-type measure on the set $B_N$ when $N \geq 3$. Queffélec-Ramaré \cite{QetR} from 2003 later managed to extend Kaufman's work to case $N = 2$ using more refined information on the non-concentration properties of the distortions.

In \cite{JordanSahlsten} from 2016 with Jordan we adapted the approach of Kaufman, Queffélec and Ramaré to prove a result that applies to all self-conformal measures on $B_N$ by adapting ideas from thermodynamical formalism and large deviation theory. 

\begin{theorem}\label{thm:JordanSahlsten}
Let $\mu_\phi$ be a stationary Gibbs measure for the IFS $\Phi = \{\frac{1}{x+a} : a \in \N\}$ associated to any locally H\"older continuous potential $\phi$ (see \cite{JordanSahlsten} for a definition) such that $\mu_\phi$ has Hausdorff dimension at least $1/2$ with the tail assumption $\mu_\phi([0,\frac{1}{a}]) = O(a^{-\gamma})$ for some $\gamma > 0$ as $a \to \infty$ along integers. Then there exists $\alpha > 0$ such that 
$$|\widehat{\mu}_\phi(\xi)| = O(|\xi|^{-\alpha})$$ 
as $|\xi| \to \infty$. 
\end{theorem}

A particular example of a self-conformal measure $\mu$, where Theorem \ref{thm:JordanSahlsten} applies to, is defined by the probability vector $p_a = 2^{-a}$, $a \in \N$. In this case the associated self-conformal measure $\mu$ has a special meaning in the sense that the distribution function $\mu(0,x]$ is the Minkowski's question mark function $?(x)$ (see e.g. \cite{Salem} for discussion on this function), and the Fourier transform of $\mu$ equals the Fourier-Stieltjes transform of $?(x)$. Since $?(x)$ as a ``slippery'' Devil's staircase exhibits smoothness, in 1954 Salem asked in \cite{Salem} whether the Fourier-Stieltjes transform of $?(x)$ decay at infinity, that is, whether the measure $\mu$ is Rajchman or not. Since $\mu$ has Hausdorff dimension $> 1/2$, Theorem \ref{thm:JordanSahlsten} implies an affirmative answer to Salem's question and gives a polynomial rate.

Let us give a rough idea of this approach in the case $N = 2$ of Theorem \ref{thm:JordanSahlsten}.

\begin{proof}[Sketch of the proof of Theorem \ref{thm:JordanSahlsten}] Let $\mu$ be a non-atomic self-conformal measure with equal weights $1/2$ for the IFS 
$$\Phi_2 = \{x \mapsto f_1(x) = 1/(x+1) , x \mapsto f_2(x) = 1/(x+2) \}.$$ 
Compared to the self-similar case, recalling \eqref{eq:FourierBound}, we can iterate $\mu = \frac{1}{2}f_1(\mu) + \frac{1}{2}f_2(\mu)$ up to some $n = c \log |\xi|$, for some $c > 0$, to obtain

\begin{align} \label{eq:FourierBoundGauss} |\widehat{\mu}(\xi)| \lesssim \int \Big|\frac{1}{2^n} \sum_{ \a \in \{1,2\}^n } e^{-2\pi i \xi f_\a(x)}\Big| \, d\mu(x),
\end{align}
where $f_\a(x) = f_{a_1} \circ \dots \circ f_{a_n}(x)$. Then fixing $\eps > 0$, in \eqref{eq:FourierBoundGauss} we consider only the ``regular'' words $\cR_n(\eps)$ formed of those $\a \in \{1,2\}^n$ for which 
$$e^{-(\lambda + \eps) n} \lesssim |f_\a'(x)| \lesssim e^{-(\lambda - \eps) n}$$ 
for all $x \in [0,1]$, where $\lambda = \int \log |G'(x)| \, d\mu(x)$ is the Lyapunov exponent associated to the IFS $B_2$. By the large deviation principle, we know that most words $\a$ satisfy this property, so in particular as $n = c \log |\xi|$ the bound for $ |\widehat{\mu}(\xi)|$ in \eqref{eq:FourierBoundGauss} is bounded by 
$$\int |S_\xi(x)| \, d\mu(x) + O(|\xi|^{-\delta})$$ 
for some $\delta > 0$, where 
$$S_\xi(x) := \frac{1}{2^n} \sum_{ \a \in \cR_n(\eps) } e^{-2\pi i \xi f_\a(x)}.$$ 
Compare this to the stopping time bound in the self-similar case \eqref{eq:FourierBoundStop}. 

Then to go from the self-conformal measure $\mu$ to Lebesgue measure, we use an idea motivated by the \textit{large sieve inequalities} in analytic number theory (e.g. (4) in \cite{BakerFirst}), exploiting the Frostman regularity of $\mu$ enable us to bound $\int |S_\xi(x)| \, d\mu(x)$ in terms of the $L^\infty$ norm of $S_\xi'$ and $L^2$ norm of $S_\xi$ in terms of the Lebesgue measure: 
$$\int |S_\xi(x)| \, d\mu(x) \lesssim \|S_\xi\|_2^2 \|S_\xi'\|_\infty^{\eps_1} |\xi|^{\eps_3} + |\xi|^{-\eps_3}$$ 
with suitable choices of constants $\eps_1,\eps_2,\eps_3 > 0$. Here $ \|S_\xi'\|_\infty$ grows in $|\xi|$, so in order to get Fourier decay, we need to ensure $\|S_\xi\|_2^2$ decays rapidly enough. This is where the non-linearity of the inverse branches manifests itself.

By expanding the square and using the triangle inequality, we can bound
$$\|S_\xi\|_2^2 \leq 2^{-2n} \sum_{\a \in \cR_n(\eps) }\sum_{\b \in \cR_n(\eps)}\Big| \int e^{-2\pi i \xi (f_\a(x) - f_\b(x))} \, dx \Big|.$$ 
It is interesting to compare this to the self-similar case, namely \eqref{eq:FourierBoundStop} achieved using the stopping time and Cauchy-Schwarz, where one studied the phases of the form $f_\a(x) - f_\a(y)$ for the \textit{same} word $\a$ but \textit{pairs} of $x$ and $y$, we now use \textit{same} $x$, but \textit{pairs} of words $\a$ and $\b$. This creates a conceptual difference in the proofs. Thus we can reduce the problem of bounding $\|S_\xi\|_2^2$ to decay of oscillatory Lebesgue integrals with the \textit{non-linear} phase 
$$x \mapsto \psi_{\a,\b}(x) := f_\a(x) - f_\b(x),$$ which we can control using a kind of stationary phase approximation, which requires us to bound from below the derivative of the phase function $\psi_{\a,\b}(x)$. This is the first time the unique properties of the Gauss map appear. Crucially we have:
$$\psi_{\a,\b}'(x) = \phi_{\a,\b}(x) ((q_{n-1}(\a) - q_{n-1}(\b)) x + q_n(\a) - q_n(\b)), \quad x \in (0,1),$$
where 
$$\phi_{\a,\b}(x) := (-1)^{n} \frac{(q_{n-1}(\a) + q_{n-1}(\b))x + q_{n}(\a) + q_{n}(\b)}{(q_{n-1}(\a) x + q_{n}(\a) )^2(q_{n-1}(\b) x + q_{n}(\b))^2}$$ 
and $q_n(\a)$ is the $n$th \textit{continuant} associated to the word $\a$, that is, writing the continued fraction expansion associated to $\a$, we obtain a rational number $p_n(\a)/q_n(\a)$, and the continuant is the denominator of this rational number. Thus $q_n(\a) \neq q_n(\b)$ and $q_{n-1}(\a) \neq q_{n-1}(\b)$, depending on how large $|q_{n}(\a) - q_{n}(\b)|$ is compared to $|q_{n-1}(\a) - q_{n-1}(\b)|$, we can use the integration by parts to bound 
$$\Big| \int e^{-2\pi i \xi (f_\a(x) - f_\b(x))} \, dx \Big|$$ 
by 
$$|\xi|^{\eps_4-1/2} |q_{n-1}(\a) - q_{n-1}(\b)|^{-1/2} \quad \text{ or } \quad |\xi|^{\eps_4-1/2} |q_{n}(\a) - q_{n}(\b)|^{-1/2}$$ for some $\eps_4 > 0$.

To then get enough decay when summing over $\a$ and $\b$, we just need to bounds for $|q_{j}(\a) - q_{j}(\b)|$ for $j = n-1,n$ from below and the number of pairs $(\a,\b)$ for which $q_{n-1}(\a) = q_{n-1}(\b)$ and $q_n(\a) = q_n(\b)$. The latter is trivial: it can only happen when $\a = \b$. In the former case, just because the continuants are integers, we always have $|q_{n}(\a) - q_{n}(\b)| \geq 1$ or $|q_{n-1}(\a) - q_{n-1}(\b)| \geq 1$. These easy bounds are already enough to prove polynomial Fourier decay for self-conformal measures $\mu$ with Hausdorff dimension greater than $2/3$. But performing more careful analysis of the cardinality of the set of pairs $(\a,\b)$ for which 
$$|q_n(\a) - q_n(\b)| \approx e^{k \lambda/2}, \quad k = 1,2,\dots,n/2$$ 
for some $\eps_5 > 0$ (which requires us to use the property that $q_{n}(\a) = q_n(\overline{\a})$ and  $q_{n-1}(\a) = p_{n-1}(\overline{\a})$ for the \textit{reverse} $\overline{\a} = a_n \dots a_1$ of $\a = a_1\dots a_n$), following an argument due to Queffélec and Ramaré \cite{QetR}, we can improve the decay allowing us to include cases of $\mu$ with dimension near $1/2$. This includes the set $B_2$.
\end{proof}

The decay property of measures $\mu$ supported on $B_N$ have an interesting connection to the \textit{Littlewood conjecture} stating that 
$$\liminf_{q \to \infty} q \|qx\|\|qy\| = 0$$ 
for all $x,y \in [0,1]$, where $\|z \| = \dist(z,\Z)$. In \cite{PoVe} from 2000 Pollington and Velani showed using the measures by Kaufman, Queffélec and Ramaré on $B_N$ that if 
$$x \in \mathbf{B} = \bigcup_{N \in \N} B_N = \{x \in [0,1] : \liminf_{q \to \infty} q\|qx\| > 0\},$$ 
then there exists $G_x \subset \mathbf{B}$ of full Hausdorff dimension such that for any $y \in G_x$, we have $q\|qx\|\|qy\| \leq (\log q)^{-1}$ for infinitely many $q$. Now, there have been much interest to study \textit{inhomogeneous} versions of these questions and here one would need to study Fourier transforms of measures supported on inhomogeneous analogue of badly approximable numbers 
$$\mathbf{B}_\gamma(c) := \{x \in [0,1] : q\|qx - \gamma\| > c\}$$ 
for fixed $\gamma \in [0,1]$ and a small $c > 0$. Following is a problem from \cite[Claim 1]{PVZZ} in 2022:

\begin{problem}\label{prob:inhom}
For any $\gamma \in [0,1]$ and small enough $c > 0$, the set $\mathbf{B}_{\gamma}(c)$ supports a probability measure $\mu$ such that $\widehat{\mu}$ decays polynomially.
\end{problem} 

The main obstacle to proving Problem \ref{prob:inhom} is that unlike $B_N$, the inhomogeneous versions are not clearly attractors of the same IFS that describes $B_N$. Potentially they can be described with such attractors to some other IFS, which might make it possible to adapt the methods of Kaufman and Queffélec-Ramaré.

\begin{remark} If we go below dimension $1/2$ in Theorem \ref{thm:JordanSahlsten}, the methods of Kaufman and Queffélec-Ramaré break down, and we end up having bad estimates from the stationary phase bounds. This highlighted that some deeper multiscale analysis ideas for $p_n(\a)/q_n(\a)$ were needed that go beyond the Kaufman and Queffélec-Ramaré type method to bound Fourier transforms in the non-linear case. A new method was introduced in the groundbreaking work of Bourgain and Dyatlov \cite{BD1} from 2017 based on the \textit{discretised sum-product theorem}, which we will discuss in the next section. We highlight that C. Stevens in his PhD thesis \cite{Stevens} from 2022, adapting the method of Bourgain and Dyatlov, was able to remove the dimension $1/2$ assumption on Theorem \ref{thm:JordanSahlsten}. Thus Theorem \ref{thm:JordanSahlsten} holds without any assumption on the Hausdorff dimension of $\mu$ except positivity.
\end{remark}

\subsection{Additive combinatoric approach and Fuchsian groups}

In 2017 in \cite{BD1} Bourgain and Dyatlov introduced a method to study the Fourier transforms of measures based on additive combinatorics, which in particular exploits the sum-product phenomenon. This method was initially introduced in the study of the Fourier transforms of the Patterson-Sullivan measures for limit sets of Fuchsian groups, which arise from subshifts of non-linear IFS formed of fractional linear transformations. However, the method has also been very successful in other contexts, such as exponential mixing in the work of Li and Pan \cite{LiPan} (see also \cite{LPX}), in the study of Fourier transforms of stationary measures for random walks on groups starting from the breakthrough work of Li \cite{Li} in 2018, see also \cite{Li2,JialunNaudPan}, and bounding the Fourier transform of self-conformal measures in general suitably non-linear iterated function systems \cite{SahlstenStevens,BakerSahlsten}, which we will discuss later.

At its heart, the method of Bourgain and Dyatlov rely on the following Fourier decay theorem for multiplicative self-convolutions $\mu^{\odot k}$ of Frostman measures $\mu$ (i.e. $\mu(B(y,r)) \lesssim r^{\kappa}$ for all $r > 0$, $y \in \R$ and Frostman exponent $\kappa > 0$), which was stated in Bourgain's original article \cite{Bourgain2010} from 2010 (the optimised $k > 1/\kappa$ was recently discovered by Orponen, de Saxc\'e and Shmerkin \cite{OSS} in 2023).

\begin{theorem}\label{thm:discretised}
For all $\kappa > 0$, there exists $\eps_1,\eps_2 > 0$ such that the following holds. Let $\mu$ be any probability measure on $[1/2,1]$ and $N \in \N$ large enough. Assume $\mu(B(y,\rho)) \lesssim \rho^{\kappa}$ for all $\rho \in [1/N,1/N^{\eps_1}]$ and $y \in \R$. Then for all $k > 1/\kappa$, $\xi \in \R$, $|\xi| \approx N$, we have
$$|\widehat{\mu^{\odot k}}(\xi)| = \Big|\int e^{-2\pi i \xi x_1 \dots x_k} \, d\mu(x_1) \dots d\mu(x_k)\Big| \lesssim |\xi|^{-\eps_2}.$$
\end{theorem}

Theorem \ref{thm:discretised} is an improvement of the phenomenon 
$$|\widehat{\mu \odot \mu}(\xi)| \lesssim |\xi|^{\frac{1}{2} - \kappa}, \quad 1 \leq |\xi| \leq N$$ 
as long as $\mu(B(y,\rho)) \lesssim \rho^{\kappa}$ for all $y \in \R$ and $\rho \in [1/N,1]$. By iterating the multiplicative convolution,  the discretised sum-product theorem implies that $\widehat{\mu^{\odot k}}$ will have more smoothness than $\mu$ (see e.g. \eqref{eq:sumprodbound} below), which then leads to Theorem \ref{thm:discretised}. Let us now give a rough outline for the proof of Theorem \ref{thm:discretised}, where more details are given in Lemmas 8.11 and Lemma 8.43 of \cite{Bourgain2010}. 

\begin{proof}[Sketch of the proof of Theorem \ref{thm:discretised}]
By taking a suitable mollifier $\tau_N$ for $N \approx |\xi|$ where the additive convolution $\tau_N \ast \mu$ satisfies  $\|\tau_N \ast \mu\|_\infty \lesssim N^{1-\kappa}$ and an integer $r \geq 1$, we can bound 
\begin{align}\label{eq:sumprodboundinitial}|\widehat{\mu^{\odot (k+1)}}(\xi)|^r \leq \int |\widehat{\mu^{\odot k}}(\eta \xi)|^r \, d\mu(\eta) \approx \int |\widehat{\mu^{\odot k}}(\eta \xi)|^r (\tau_N \ast \mu)(\eta) \, d\eta.
\end{align}
Thus, as long as we can find integers $k,r \geq 2$ such that
\begin{align}\label{eq:sumprodbound}
\int_{-N}^N |\widehat{\mu^{\odot k}}(\eta)|^r \, d\eta < N^{\kappa/2},
\end{align}
integration by parts with respect to Lebesgue measure produces a term $|\xi|^{-1}$ in front on the right-hand side of \eqref{eq:sumprodboundinitial}, by $\|\tau_N \ast \mu\|_\infty \lesssim N^{1-\kappa}$ we get the required estimate for $|\widehat{\mu^{\odot k}}(\xi)|$ with $k$ replaced by $k+1$. We note in \cite{OSS}, that improved \cite{Bourgain2010}, is that $r$ can be made equal to $2$ in \eqref{eq:sumprodbound}, which would allow us to optimise $k$ in terms of $\kappa$.

To find $k$ and $r$ such that \eqref{eq:sumprodbound} holds one has to study the numbers $\sigma_{k,r} > 0$ defined by the expression 
$$\int_{-N}^N |\widehat{\mu^{\odot k}}(\eta)|^r \, d\eta = N^{\sigma_{k,r}}.$$
The numbers $\sigma_{k,r}$ are essentially decreasing both in $k$ and $r$, so one looks for large enough $k$ and $r$ that we can use. It turns out that the discretised sum-product theorem \cite{Bourgain2010} gives the existence of $\eps_1 > 0$ such that $\sigma_{2k,\overline{r}} < \sigma_{k,2r} - \eps_1/2$ for a suitable choice $\overline r \approx c_3 r^2$, where $c_3 > 0$ is chosen appropriately. Thus iterating this with respect to $k$ and $r$, we can find $k$ and $r$ such that $\sigma_{k,2r} < \kappa/2$ giving \eqref{eq:sumprodbound}. 

We can achieve $\sigma_{2k,\overline{r}} < \sigma_{k,2r} - \eps_1/2$ by studying measures
$$\nu_{r,k} := \frac{1}{2}[(\mu^{\odot k} \ast \mu^{\odot k}_{-})^{\ast r} + (\mu^{\odot 2k} \ast \mu^{\odot 2k}_{-})^{\ast r}],$$ 
where $\mu^{\ast r}$ is the additive self-convolution of $\mu$ and $\mu_{-}$ is the push-forward of $\mu$ under $x \mapsto -x$. Then 
$$\int_{-N}^N |\widehat{\nu}_{r,k}(\eta)|^2 \, d\eta \sim \int_{-N}^N |\widehat{\mu^{\odot k}}(\eta)|^{2r} \, d\eta = N^{\sigma_{k,2r}}.$$ 
Thus if we can find $\eps_1 > 0$ such that there is an $N^{-\eps_1}$ gain:
\begin{align}\label{eq:sumpr}\int_{-N}^N \int |\widehat{\nu}_{r,k}(\eta)|^2 |\widehat{\nu}_{r,k}(y \eta)|^2 d\nu_{r,k}^{\ast 2}(y) \, d\eta \lesssim N^{-\eps_1} \int_{-N}^N |\widehat{\nu}_{r,k}(\eta)|^2 \, d\eta \qquad (\sim N^{-\eps_1 + \sigma_{k,2r}}),\end{align}
then with some effort one can deduce that this estimate gives us $\sigma_{2k,\overline{r}} < \sigma_{k,2r} - \eps_1/2$ by using a covering argument on points $\eta \in \R$ with $|\widehat{\mu}_k(\eta)| > N^{-\kappa_0}$ for $\kappa_0 = 2 c_3^{-1} r^{-2}$.

The estimate \eqref{eq:sumpr} follows from Bourgain's discretised sum-product theorem \cite{Bourgain2010} from 2010. Set $\nu := \nu_{r,k}$. If \eqref{eq:sumpr} fails for $\nu$, then any small $\eps > 0$, we have 
$$\int_{|y| > N^{-\eps_1}}|S \cap y^{-1} S| \, d\nu^{\ast 2}(y) > N^{\sigma_{k,2r}-\eps}$$ 
for a scale $N^{-1}$ discretisation $S = \bigcup_{j = 1}^M B(z_j,N^{-1-\eps})$ of the support of $\nu^{\ast 2}$ for some $1/N$ separated points $z_j$ in the support of $\nu^{\ast 2}$. Thus we can find many $1/N$ separated points $y_k$ from the support of $\nu$ such that $|z_{j_1} - y_k z_{j_k}|$ is large for many pairs $(z_{j_1},z_{j_2})$. This can then be turned into a lower bound on the number of quadruples $(\log |z_{j_1}|,\log |z_{j_2}|,\log |z_{j_3}|,\log |z_{j_4}|) \in (\log |S|)^4$ for which 
$$\log |z_{j_1}|+\log |z_{j_2}| = \log |z_{j_3}| + \log |z_{j_4}|.$$ In other words, $\log |S|$ has a large \textit{additive energy}, so by the Balog-Szemer\'edi-Gowers theorem (see e.g. \cite{Gowers}) there exists a large subset $\cF_1$ of $S$ such that $\cF_1 \cdot \cF_1$ can be covered by at most $N^{\eps_3}|\cF_1|$ intervals of size $1/N$ for a suitable $\eps_3> 0$. On the other hand, as $S$ comes from the support of the additive self-convolution $\nu^{\ast 2}$, by summing a mollifier of $\nu^{\ast 2}$ over points in $\cF_1 \subset S$, Cauchy-Schwarz implies that $\cF_1$ also has large additive energy, so again by Balog-Szemer\'edi-Gowers theorem we can find a further large subset $\cF_2 \subset \cF_1$ such that $\cF_2 + \cF_2$ can be covered by at most $N^{\eps_4}|\cF_2|$ intervals of size $1/N$ for a suitable $\eps_4 > 0$. As the sets $\cF_2 \subset \cF_1$ are large subsets of $S$, this is impossible by the discretised sum-product theorem \cite{Bourgain2010}.
\end{proof}

In \cite{BD1}, Bourgain and Dyatlov applied Theorem \ref{thm:discretised} to the study of Fourier transforms of stationary measures on \textit{limit sets} $\Lambda_\Gamma$ of convex and co-compact Fuchsian groups $\Gamma$, which like the Gauss map have a non-linear structure as they can be described using contractions given by fractional linear transformations acting on the unit circle. Here Bourgain and Dyatlov studied the \textit{Patterson-Sullivan measure} on the limit set $\Lambda_\Gamma$, which is a natural stationary measure associated to a subshift of a non-linear IFS arising from the group $\Gamma$.

\begin{theorem}\label{thm:bourgaindyatlov}
If $\mu_\Gamma$ is the Patterson-Sullivan measure on the limit set $\Lambda_\Gamma$ associated to a convex co-compact Fuchsian group $\Gamma \subset \mathrm{PSL}(2,\R)$, then there exists $\alpha = \alpha(\Hd \Lambda_\Gamma) > 0$ such that 
$$|\widehat{\mu}(\xi)| = O(|\xi|^{-\alpha})$$ 
as $|\xi| \to \infty$.
\end{theorem}

The motivation of Theorem \ref{thm:bourgaindyatlov} was to study \textit{Fractal Uncertainty Principles} for the limit set $\Lambda_\Gamma$ when $\Hd \Lambda_\Gamma \leq \frac{1}{2}$, which we recall here. Two sets $X,Y \subset \R$ are said to satisfy the \textit{Fractal Uncertainty Principle} (FUP) with exponent $\beta > 0$ if for all $0 < h \ll 1$, whenever $f \in L^2(\R^d)$ satisfies $\supp \widehat{f} \subset \{y/h : \mathrm{dist}(y,Y) \leq h\}$, then 
$$\int_{X(h)} |f(x)|^2 \, dx \lesssim h^{2\beta},$$
where $dx$ is the Lebesgue measure on $\R$. In particular Bourgain and Dyatlov \cite{BD1} in 2017 connected the Fourier decay rate exponent of the Patterson-Sulllivan measure to the FUP exponent $\beta > 0$ applied to $X = Y = \Lambda_\Gamma$. Adapting Dolgopyat's method \cite{Dolgopyat} from 1998, Dyatlov and Jin \cite{DJ2} in 2017 proved a bound $\beta = \frac{1}{2} - \Hd \Lambda_\Gamma + \alpha$, for some $\alpha > 0$, that depended on the Ahlfors-David regularity constants of the limit set $\Lambda_\Gamma$. Bourgain and Dyatlov \cite{BD1} in 2017 were able to exploit the nonlinearity of the IFS in a way that $\alpha$ only depended on $\Hd \Lambda_\Gamma$. This, combined with an earlier work of Dyatlov and Zahl \cite{DyatlovZahl} from 2016 led to a bound for the \textit{essential spectral gap} of the Laplacian operator on the corresponding non-compact hyperbolic surface $\Gamma \setminus \mathbb{H}$ that only depended on the Hausdorff dimension of the limit set.

Let us now sketch how Theorem \ref{thm:discretised} appears in proving Theorem \ref{thm:bourgaindyatlov}. 

\begin{proof}[Sketch of the proof of Theorem \ref{thm:bourgaindyatlov}] When $\Gamma$ is a convex co-compact subgroup, it has an associated \textit{Schottky structure} \cite[Section 15.1]{Bor}, which allows us to code the limit set $\Lambda_\Gamma$ using matrices $\gamma_\a \in \mathrm{SL}(2,\R)$, which we can identify as M\"obius transformations acting on the hyperbolic plane $\mathbb{H}$ and contracting real analytic maps on the boundary of $\mathbb{H}$. Here $\a \in \cW$, where $\cW \subset \bigcup_{n \in \N} \cA^n$ is the set of allowed finite words from some index set $\cA$ that satisfy rules defined by the Schottky structure. This also allows us to realise the Patterson-Sullivan measure $\mu_\Gamma$ as a Gibbs measure associated to a subshift of an IFS.

Instead of now trying to iterate transfer operator up to words of length $n$ as in the case of Gauss map (recall \eqref{eq:FourierBound}), Bourgain and Dyatlov \cite{BD1} instead apply \textit{blocks} of length $k$ and $k+1$ formed of words of lengths $n$. The $k$ is precisely the $k$ that will arise from Theorem \ref{thm:discretised}. Now, fix a suitable $\eps_1 > 0$ such that $\tau \approx |\xi|^{-\eps_1}$ and words $\cW(\tau) \subset \cW$ such that any $\a \in \cW(\tau)$ satisfy that the lengths $|\gamma_\a[0,1]| \approx \tau$. Then, as $\mu_\Gamma$ is Ahlfors-David regular, Cauchy-Schwarz inequality (compare to \eqref{eq:FourierBoundStop} in the self-similar case), mean value theorem and the fact that the contractions $\gamma_\a$ exhibit bounded distortions, there exists suitable constants $\eps_2 > 0$ and $C > 0$ such that

\begin{align} \label{eq:FourierBoundFuchsian} |\widehat{\mu}_\Gamma(\xi)|^2 \lesssim \tau^{k\delta} \max_{\A =\a_1 \dots \a_{k+1}} \sup_{\tau^{-1/4} \leq |\eta| \leq C \tau^{-1/2}} \Big| \sum_{\b_1 \dots \b_{k} \in \cB(\A)} e^{-2\pi i \eta \prod_{j = 1}^k \tau^{-2} \gamma'_{\a_{j-1}'\b_j'}(x_{\a_j})}\Big| + \tau^{\eps_2},
\end{align}
where $\cB(\A)$ are the collection of blocks of words in $\cW(\tau)$ that are allowed in the Schottky structure depending $\A$, words $\a = a_1\dots a_n$ and we use the notations $\a' = a_1\dots a_{n-1}$ and $x_\a$ as the center point of $\gamma_\a[0,1]$. Now we see that the right hand-side of \eqref{eq:FourierBoundFuchsian} is a multiplicative convolution of the discrete measures 
$$\nu_j := \sharp \cZ_j^{-1}\sum_{\b \in \cZ_j} \delta_{\tau^{-2} \gamma_{\a_j' \b'}' (x_\a)}$$ 
for some finite sets $\cZ_j$. Thus by adapting Theorem \ref{thm:discretised} to multiplicative convolutions of several $\nu_j$, having polynomial Fourier decay bound reduces by \eqref{eq:FourierBoundFuchsian} to proving the Frostman condition required by Theorem \ref{thm:discretised} for each $\nu_j$. This means that we need to prove that the normalised derivatives $\tau^{-2} \gamma_{\a_j' \b'}' (x_\a)$, $\b \in \cW(\tau)$, do not concentrate near a point.

Let us now argue the non-concentration of the numbers $\tau^{-2} \gamma_{\a_j' \b'}' (x_\a)$, $\b \in \cW(\tau)$. Fixing $b \in \cA$, we write $I_b = [x_0,x_1]$. Then we can write $\gamma_\a(x) = S_b \gamma_{\alpha_{\a,b}}T_{\a,b}^{-1}$, where $T_{\a,b}$ and $S_b$ are affine maps such that $T_{\a,b}([0,1]) = \gamma_\a(I_b)$ and $S_b([0,1]) = I_b$ and
$$\gamma_{\alpha_{\a,b}} = \begin{pmatrix} e^{\alpha_\a/2} & 0 \\ e^{\alpha_{\a,b}/2} - e^{-\alpha_{\a,b}/2} & e^{-\alpha_{\a,b}/2}\end{pmatrix}, \quad \text{where}\quad \alpha_{\a,b} = \log \frac{\gamma_\a^{-1}(\infty) - x_1}{\gamma_\a^{-1}(\infty) - x_0}$$
is the \textit{distortion} of $\gamma_\a$ in $I_b$. Thus 
$$\gamma_\a'(x) = \gamma_{\alpha_{\a,b}}'(T_{\a,b}^{-1}(x)) \cdot \frac{|\gamma_\a(I_b)|}{|I_b|}$$ 
so using bounded distortions we can reduce the non-concentration properties of $\tau^{-2} \gamma_{\a_j' \b'}' (x_\a)$ to non-concentrations of $\gamma_{\a_j' \b'}^{-1}(\infty)$. As with the Gauss map, if we consider the \textit{reverse} word $\overline{\a}$ of $\a$, the lengths 
$$|\gamma_\a[0,1]| \approx |\gamma_{\overline{\a}}[0,1]|.$$ 
The Schottky structure was defined so that $\gamma_{\a_j' \b'}^{-1}(\infty) = \gamma_{\overline{\a_j' \b'}}(\infty) \in \gamma_{\overline{\b'}}[0,1]$, which has a length approximately the same as $|\gamma_{\b'}[0,1]| \approx \tau$. Given $\a_j$, $y \in \R$ and $\rho > \tau$, we can thus cover $B(y,\rho) \cap \{\gamma_{\a_j' \b'}^{-1}(\infty) : \b \in \cW(\tau)\}$, by the disjoint intervals $\gamma_{\b'}[0,1]$ that intersect $B(y,\rho)$. This implies an upper bound for the number of $\gamma_{\a_j' \b'}^{-1}(\infty)$, $\b \in \cW(\tau)$, with $\gamma_{\a_j' \b'}^{-1}(\infty) \in B(y,\rho)$.
\end{proof}

\subsection{Renewal theoretic approach and random walks on groups} 

Briefly after the work of Bourgain and Dyatlov came out, Li \cite{Li} in 2018 introduced a novel approach to the Fourier decay problem based on quantitative \textit{renewal theorems} for random walks on Lie groups. In this work Li considered the stationary measures and related the Fourier coefficients to renewal operators, which, using the non-linearity properties of the action of the maps, gave Fourier decay. Li's approach also applies in the case of Patterson-Sullivan measures considered by Bourgain and Dyatlov, but also opens the door for other proofs such as the self-similar case (which we discussed in the earlier section), and an approach to general non-linear self-conformal measures and self-affine measures, which we will describe in the next sections.

In \cite{Li} Li considered a probability distribution $\nu$ of matrices from $\mathrm{SL}(2,\R)$ that has \textit{finite exponential moment} $\int \|g\|^{\eps_1} \, d\nu(g) < \infty$ for some constant $\eps_1 > 0$. Then a key non-concentration assumption one needs for the subgroup $\Gamma_\nu$ generated by the support of $\nu$ is \textit{Zariski density} of $\Gamma_\nu$ in $\mathrm{SL}(2,\R)$. This roughly speaking means that $\nu$ does not concentrate mass on any proper algebraic subgroup of $\mathrm{SL}(2,\R)$. Now the random walk generated by $\nu$ acts on the projective space $\mathbb{P}^1$ by M\"obius transformations. By Zariski density of $\Gamma_\nu$, there exists a unique $\nu$-stationary measure $\mu$ on $\mathbb{P}^1$, that is, satisfying $\mu = \nu \ast \mu$, known as the \textit{Furstenberg measure} \cite{BQ}.

In \cite{Li} Li proved under these assumptions that the Furstenberg measure $\mu$ is always Rajchman adapting an approach that did not rely on the sum-product phenomenon.

\begin{theorem}\label{thm:li}
Let $\nu$ be a Borel probability measure on $\mathrm{SL}(2,\R)$ with a finite exponential moment such that the subgroup $\Gamma_\nu$ generated by the support of $\nu$ is Zariski dense. Then the $\nu$-stationary measure $\mu$ on $\mathbb{P}^1$ is Rajchman and there exists $\alpha > 0$ such that 
$$|\widehat{\mu}(\xi)| \lesssim |\xi|^{-\alpha}$$
as $|\xi| \to \infty$.
\end{theorem}

We remark that in \cite{Li} only the Rajchman property was proved and the polynomial rate of decay was established in a later paper \cite{Li2} published in 2018 by Li, where the sum-product bound was adapted to prove in Theorem \ref{thm:li} the rate of Fourier decay is polynomial. The idea of the proof of Theorem \ref{thm:li} is to use Cauchy-Schwarz and a stopping time to relate $|\widehat{\mu}(\xi)|^2$ to a renewal operator associated to a stopping time $n_t$ (a residue process), which is the smallest $n \in \N$ such that $\log \|X_n\| \geq t$, where $X_n = g_1\dots g_n$ is a random walk, where $g_j$ are i.i.d. and distributed according to $\nu$. The idea was adapted in the self-similar case by Li and the author \cite{JialunSahlsten1} in 2019, recall  \eqref{eq:FourierBoundStop}, where the idea is easier to follow. Then a renewal theorem of Kesten \cite{Kesten1974} from 1974 can be applied to give some limit distribution for the residue process, but to get Fourier decay, uniform rate of convergence is needed, which Li establishes in \cite{Li}.

The polynomial rate of decay in Theorem \ref{thm:li} leads also to a uniform spectral gap bound to transfer operators associated to the random walk $X_n = g_1\dots g_n$ on $\mathrm{SL}(2,\R)$ generated by $\nu$ and an \textit{exponential} error term rate in the renewal theorem that leads to polynomial Fourier decay. Li's work \cite{Li2} also applies to many examples coming from higher dimensions, which we will discuss in a later section. 

If one removes the exponential moment condition, and only uses a finite second moment for $\nu$, that is, $\int \log^2 \|g\|\, d\nu(g) < \infty$, it is still possible to establish Fourier decay for the $\nu$-stationary measure $\mu$, and a log-log-rate if this moment condition is strengthened slightly. This was proved by Dinh,  Kaufmann and Wu \cite{DKW} from 2023. 

\begin{theorem}\label{thm:DihnKaufmannWu}
Let $\nu$ be a Borel probability measure on $\mathrm{SL}(2,\R)$ with a finite second moment such that the subgroup $\Gamma_\nu$ generated by the support of $\mu$ is Zariski dense. Then the unique $\mu$-stationary measure $\nu$ on $\mathbb{P}^1$ is Rajchman. Moreover, if there exists $\eps > 0$ such that $\int \log^{2+\eps} \|g\|\, dg < \infty$, then 
$$|\widehat{\mu}(\xi)| \lesssim \frac{1}{\log \log |\xi|}$$ 
as $|\xi| \to \infty$.
\end{theorem}

The work of Dinh, Kaufmann and Wu has to overcome crucial issues that arise when only low regularity is imposed, in particular, the large deviation estimates by Benoist and Quint \cite{BQ16} from 2016 that are needed in the proof of Li, do not hold.

The methods developed in the special cases of Gauss map, Fuchsian groups and the stationary measures for random walks on groups raise the question of whether one could have a classification that holds for \textit{general} self-conformal iterated function systems. This is the focus of next section.

\subsection{General non-linear iterated function systems} 

All the iterated function systems $\Phi$ that arise from the Gauss map or Fuchsian groups share the property that the branches $f_j$ are \textit{non-linear}, that is, $f_j$ are not similitudes. As many natural stationary measures for these systems have polynomial Fourier decay, while some self-similar systems do not, would there be a way to prove that the existence of nonlinearity in the IFS always implies rapid decay of the Fourier transforms? A prototype question in this direction would be:

\begin{problem}\label{prob:nonlinear}
Suppose $\Phi = \{f_j : j \in \cA\}$ is a $C^{1+\alpha}$ IFS such that one of the maps $f_j$ is not a similitude. Does the Fourier transform of every non-atomic self-conformal measure associated to $\Phi$ and probability vector $p_j \in (0,1)$, $j \in \cA$, decay polynomially?
\end{problem}

Over the past years, there has been some substantial progress towards Problem \ref{prob:nonlinear} in the literature. Problem \ref{prob:nonlinear} has also some connections in bounding oscillatory integrals with H\"older phases in the work of Leclerc \cite{LeclercOsc} from 2022, prevalence of normal numbers and Diophantine properties of numbers in the supports of $K_\Phi$ by the Davenport-Erd\"os-LeVeque criterion \cite{DEL} and \cite{PVZZ}, Fractal Uncertainty Principles in quantum chaos (see the Appendix of the work of Baker and the author \cite{BakerSahlsten} from 2023), and the study of its connections to applications in hyperbolic dynamics such as the Newhouse phenomenon in an upcoming work of Avila, Lyubich and Zhang \cite{ALZ}.

Now to attack Problem \ref{prob:nonlinear}, it would be tempting to try to adapt one of the thermodynamical, additive combinatoric or renewal theoretical approaches. The thermodynamic approach that relies on large deviation theory is a good way to deal with measures that are not Ahlfors-David regular, while the Patterson-Sullivan measures in the additive combinatoric used approach are. Moreover, as we saw in the self-similar case, the renewal theoretic approach by Li \cite{Li} from 2018 works well even with complicated overlaps in the IFS.

In \cite{SahlstenStevens} with Stevens we studied how the ideas of Jordan and the author \cite{JordanSahlsten} and Bourgain-Dyatlov \cite{BD1}, together with some ideas from random walks on groups by Li \cite{Li2}, could be combined to try to tackle Problem \ref{prob:nonlinear} in the $C^2$ category under suitable separation conditions and with a different non-linearity assumption coming from the theory of exponential mixing of Anosov flows. We studied $C^2$ IFSs $\Phi = \{f_j : h \in \cA\}$ that satisfied \textit{Uniform Non-Integrability} (UNI), that is, there exists constants $c_1,c_2 > 0$ such that for all large enough $n$, we can always find two words $\a,\b \in \cA^n$ such that the distortions satify:
$$c_1 \leq \Big|\frac{f_\a''(x)}{f_\a'(x)} - \frac{f_\b''(x)}{f_\b'(x)}\Big| \leq c_2, \quad x \in [0,1].$$
This condition was introduced by Chernov \cite{Chernov} from 1998 and studied by Dolgopyat \cite{Dolgopyat1,Dolgopyat} from 1998 as a key assumption for an IFSs coming from Anosov flows to prove the exponential mixing of the flow. In particular, if $\Phi$ satisfies the \textit{strong separation condition} (i.e. $f_j(K_\Phi)$, $j \in \cA$ are disjoint), the UNI condition, as shown by Naud \cite{Naud} in 2005 and Stoyanov \cite{Stoyanov} from 2011 in higher dimensions, allows one to prove \textit{uniform spectral gap} bounds for the complex transfer operators
$$\cL_s u(x) = \sum_{j \in \cA} p_j |f_j'(x)|^{s} u(f_j(x)), \quad x \in \R, u \in C^1(\R)$$
for $s \in \C$ with $|\mathrm{Re}(s)|$ small and $|\mathrm{Im}(s)|$ large. That is, bounds for the spectral radius of $\cL_s$ that are uniform in $s$ in this strip.

If $\Phi$ is $C^\omega$, then it goes back to the work of Naud \cite{Naud}, who showed that the UNI condition can be checked with a slightly more easily checkable \textit{total non-linearity} condition that says the IFS is not $C^1$ conjugated to an IFS $\{h_j : j \in \cA\}$ such that $h_j'$ are locally constant, that is, it is not possible to find a $C^1$ diffeomorphism $g : \R\to \R$ such that for all $j \in \cA$ there exists $a_j \in \R$ such that
$$\log |f_j'(x)| = a_j + g(f_j(x)) - g(x), \quad x \in [0,1].$$
If $\Phi$ is $C^{\omega}$, total non-linearity is equivalent to UNI, which was shown in Naud's paper \cite{Naud}. Moreover, in \cite{Naud} Naud studied a more general condition known as Non-Local Integrability showing as examples that the limit sets of Fuchsian groups and quadratic Julia sets satisfy the UNI, see also \cite{BV} for study in the case of the Gauss map. If $\Phi$ is only $C^{2}$ and the attractor $K_\Phi = [0,1]$, in which case only an open set condition is satisfied, then total non-linearity is equivalent to UNI and there is still a uniform spectral gap, see e.g. \cite{AGY,Stoyanov}.

Under the UNI condition and strong separations, we proved in \cite{SahlstenStevens} the following case:

\begin{theorem}\label{thm:StevensSahlsten}
Suppose $\Phi$ is a $C^2$ IFS satisfying UNI and the strong separation condition. Then the Fourier transform of every non-atomic self-conformal measure $\mu$ associated to $\Phi$ has polynomial Fourier decay.
\end{theorem}

To study Fourier transform of $|\widehat{\mu}(\xi)|$ we can first, as in the work of Jordan and the author \cite{JordanSahlsten}, use large deviations to just focus on the words $\a \in \cA^n$ for which 
$$e^{-(\lambda +\eps) n } \lesssim |f_\a'(x)| \lesssim e^{-(\lambda -\eps) n }$$ 
for some fixed $\eps > 0$. Then we can reduce, as Bourgain and Dyatlov \cite{BD1} did, by iterating using blocks of length $k$ and $k+1$ of length $n$ words to exponential sums involving products of normalised derivatives $e^{-2\lambda n} f_{\a_j' \b'}' (x_\a)$, and employ the sum-product theory (recall Theorem \ref{thm:discretised}). However, the difficulty is to show non-concentration of the normalised derivatives. Here is where the transfer operators $\cL_s$ with $s$ of the form $s = ic\xi$ for a suitable constant $c > 0$ can be helpful. As it was shown in Li \cite{Li2} in 2018 for transfer operators arising from random walks on groups, having uniform spectral gap for complex transfer operators can be used to prove non-concentration type estimates for the random walks Li considered. With Stevens, we followed this scheme in the proof of Theorem \ref{thm:StevensSahlsten} adapted to the transfer operators $\cL_s$.

Now going back to Problem \ref{prob:nonlinear}, what about the cases where the IFS has overlaps, has $C^{1+\alpha}$ regularity or the UNI condition fails? Considerable advances have been done recently to resolve these, which we now review. Briefly after \cite{SahlstenStevens} came out, Algom, Hertz and Wang \cite{AHW} adapted independently a different approach to the study of $C^{1+\alpha}$ IFS that can have arbitrary overlaps. In particular, they managed to deal with low regularity and overlaps at the price of non-quantitative rate of Fourier decay by introducing a new method based on quantitative central limit theorems on random walks on groups by Benoist and Quint \cite{BQ} from 2016:

\begin{theorem}\label{thm:AHW1}
Suppose $\Phi$ is a $C^{1+\alpha}$ IFS such that $\{\log |f_j'(x)| : f_j(x) = x, j \in \cA\}$ is not included in any $a\Z + b$. Then every non-atomic self-conformal measure is Rajchman.
\end{theorem}

Finding out the rate of decay (polylogarithmic or polynomial), remained as a very challenging problem especially in the low regularity and in the overlapping case due to the lack of tools in these settings. Indeed, in Theorem \ref{thm:StevensSahlsten} a spectral gap was proved that needed a separation condition where as in Theorem \ref{thm:AHW1} the central limit theorem rate could not be used to get such rapid decay even in the best cases.

If one makes the IFS more regular, some of these issues can be overcome. We say that $\Phi$ is \textit{$C^r$ conjugated to IFS} $\Psi$ if there exists a $C^r$ diffeomorphism $h : \R \to \R$ such that $h\Phi = \{h f_j h^{-1} : j \in \cA\} = \Psi$. In the literature, such as \cite{HochmanShmerkin}, $\Phi$ is also called \textit{totally non-linear} if it is not $C^1$ conjugated to a self-similar IFS, which is slightly different from the earlier definition described using the log-derivatives. A $C^1$ IFS $\Phi$ is Diophantine if there exists $l,C > 0$ such that 
$$\inf_{y \in \R} \max_{1 \leq i \leq n} d(\log |f'_j(x)| \cdot x + y,\Z) \geq C|x|^{-l}$$ 
for all $x \in \R$ with $|x|$ large enough. We say IFS $\Psi = \{g_j : j \in \cA\}$ is \textit{linear} if $g_j'' = 0$ for all $j \in \cA$. Then as long as the IFS is not conjugated to linear Diophantine IFSs, Algom, Hertz, Wang obtained:

\begin{theorem}\label{thm:AHW2}
Let $\Phi$ be an orientation preserving $C^r$ IFS, $r \geq 2$, that is not $C^r$ conjugate to a linear non-Diophantine IFS. Then every non-atomic self conformal measure associated to $\Phi$ has polylogarithmic Fourier decay.
\end{theorem}

It is helpful to compare this to \cite{JialunSahlsten1}, where polylogarithmic Fourier decay was obtained for self-similar IFSs $\{x \mapsto r_j x + b_j : j \in \cA\}$ for which $\log r_i / \log r_j$ is not Liouville number for some $i \neq j$. Note that Algom, Hertz, and Wang also showed that if a $C^{\omega}$ IFS is linear, then it is self-similar, so Theorem \ref{thm:AHW2} in the $C^\omega$ category is about not being conjugated to non-Diophantine linear IFSs like the ones considered in \cite{JialunSahlsten1}.

After Theorems \ref{thm:StevensSahlsten} and \ref{thm:AHW1} and \ref{thm:AHW2}, it remains unclear how to prove polynomial Fourier decay in the overlapping case or having quantitative rate in the $C^{1+\alpha}$ category. If we remove the separation condition in the proof of Theorem \ref{thm:StevensSahlsten}, a key obstacle in the approach used in \cite{SahlstenStevens} was that the uniform spectral gap for $\cL_s$ that was needed for the non-concentration of the normalised derivatives was not known. This was also an obstacle for adapting a renewal theoretic approach as the rate was not known to be exponential. All these were resolved by Algom, Hertz, Wang \cite{AHW3} in 2023, and simultaneously Baker and the author \cite{BakerSahlsten} in 2023, where two different approaches to the polynomial Fourier decay in the overlapping $C^2$ category were established:

\begin{theorem}\label{thm:generalnonlinearUNI}
Suppose $\Phi$ is a $C^2$ IFS that satisfies UNI. Then the Fourier transform of every non-atomic self-conformal measure $\mu$ associated to $\Phi$ has polynomial Fourier decay.
\end{theorem}

The articles \cite{AHW3, BakerSahlsten} in particular proved new uniform spectral gap bounds for $\cL_s$ near the critical line in the overlapping case that requires a co-cycle version of Dolgopyat's method. The methods to study transfer operators $\cL_s$ for overlapping IFSs was based on a disintegration of $\cL_s$ using non-linear sub-IFSs, which was inspired by the work of Algom, Baker and Shmerkin \cite{ABS} from 2022 in the study of normal numbers on self-similar sets. 

Going back to Problem \ref{prob:nonlinear}, combined with the argument of Algom, Hertz and Wang \cite{AHW,AHW2}, Theorem \ref{thm:generalnonlinearUNI} implies Problem \ref{prob:nonlinear} in the $C^\omega$ category as long as $\Phi$ is \textit{not} $C^\omega$ conjugated to a self-similar IFS. If such $C^\omega$ conjugacy exists, then the problem still remains open in its full generality. It goes back to the work of Kaufman \cite{KaufmanSurvey} from 1984, where it was proved that if $F : \R \to \R$ satisfies $F'' > 0$ everywhere, then the Fourier transform of $F \mu$ decays polynomially if $\mu$ is the middle $1/3$ Cantor measure. This was generalised to general homogeneous self-similar measures $\mu$ associated to an equicontractive self-similar IFS $\Phi$ by Mosquera and Shmerkin \cite{MS} in 2018. In particular, $F\mu$ is a self-conformal measure to an IFS that is $C^{\omega}$ conjugated to the self-similar IFS $\Phi$. 

Recently, Algom, Chang, Wu, and Wu \cite{ACWW} in 2023 and independently Baker and Banaji \cite{BakerBanaji} in 2023, generalised the works of Kaufman and Mosquera-Shmerkin to pushforwards $F\mu$ of \textit{general} self-similar measures $\nu$ under $C^2$ maps $F : \R \to \R$ with $F'' > 0$ everywhere. The methods were very different: while \cite{ACWW} adapted the large deviation principle of Tsujii \cite{Tsujii} from 2015 to overcome the lack of convolution structure, the article \cite{BakerBanaji} adapted an idea from Galcier-Saglietti-Shmerkin-Yavicoli \cite{GSSY} from 2016 (see also \cite{SSS}) of decomposing a self-similar measure to equicontractive self-similar measures in the study of $L^q$ dimensions, projections and absolute continuity of fractal measures.

When applied to $C^\omega$ IFSs, the works \cite{ACWW,BakerBanaji} lead to the following consequence:

\begin{theorem}\label{thm:nonlinearConjugated}
Suppose $\Phi$ is a $C^\omega$ IFS that is $C^\omega$ conjugated to a self-similar IFS and there exists $f_j \in \cA$ that is not a similitude. Then the Fourier transform of every non-atomic self-conformal measure $\mu$ associated to $\Phi$ has polynomial Fourier decay.
\end{theorem}

Together with Theorem \ref{thm:generalnonlinearUNI}, Theorem \ref{thm:nonlinearConjugated} implies that Problem \ref{prob:nonlinear} is fully resolved for all $C^\omega$ IFSs. Now what about $C^2$ or $C^{1+\alpha}$ IFSs? This remains an open problem, and useful in the study e.g. in Fractal Uncertainty Principle for hyperbolic sets arising from Anosov flows (see discussions in the work of Dyatlov and Zahl \cite{DyatlovZahl} from 2016). In the $C^2$ category, due to Theorem \ref{thm:generalnonlinearUNI}, one would need to find a way to classify IFSs that do not satisfy UNI:

\begin{problem}\label{prob:nonlinearConjugacyC2}
Suppose $\Phi = \{f_j : j \in \cA\}$ is a $C^{2}$ IFS that does not satisfy UNI and one of the maps $f_j$ is not a similitude. Does the Fourier transform of every non-atomic self-conformal measure associated to $\Phi$ decay polynomially?
\end{problem}

In the $C^{1+\alpha}$ category, things become much more difficult and only the Rajchman property is known by Theorem \ref{thm:AHW1}. The difficulty of proving spectral gap bounds lies in the fact that the way e.g. Naud \cite{Naud} in 2005 adapts Dolgopyat's method replies on requiring the study of second derivative at many points to be able to use the UNI condition. However, there are some recent advances to overcome this in the proof of the exponential mixing of topologically mixing $C^\infty$ Anosov flow on $3$-dimensional compact manifolds by Tsujii and Zhang \cite{TZ} in 2023. In \cite{TZ} the UNI condition was replaced a kind of oscillation condition on the temporal distance function coming from the flow, which could have an analogue in the symbolic setting. Thus to generalise Theorem \ref{thm:AHW1}, one could potentially try to attack Problem \ref{prob:nonlinear} using these new methods and check first e.g. polylogarithmic Fourier decay:

\begin{problem}\label{prob:generalnonlinearCalpha}
Suppose $\Phi$ is a $C^{1+\alpha}$ IFS that is not $C^{1+\alpha}$ conjugate to a linear non-Diophantine IFS. Then does the Fourier transform of every non-atomic self-conformal measure associated to $\Phi$ decay polylogarithmically?
\end{problem}

During the writing of this manuscript, it came to the attention of the author that the upcoming work of Avila, Lyubich and Zhang \cite{ALZ} addressed Problem \ref{prob:generalnonlinearCalpha} under suitable additional conditions. Here Avila, Lyubich and Zhang managed to prove a polylogarithmic Fourier decay of self-conformal measures with a \textit{very large} exponent under a suitable non-linearity assumption on the IFS. Their work was motivated by the study of the Newhouse phenomenon for the H\'enon family related to the \textit{Palis' conjectures} (see e.g. \cite{Matheus} for an overview) and resulting intersections of non-linear Cantor sets. 

Finally, it would be very interesting to try to \textit{optimise} the rate of polynomial Fourier decay of self-conformal measures, and potentially prove the Salem property of the attractor $K_\Phi$, recall the numerics from Figure \ref{fig:1}. Optimising the rate of Fourier decay would lead to advances towards the conjecture of Jacobson and Naud \cite{JN} in 2010 which predict an \textit{optimal essential spectral gap} in the setting of convex co-compact hyperbolic surfaces when the dimension of the limit set is at most $1/2$. Indeed, with a Fractal Uncertainty Principle argument \cite{BD1,DyatlovZahl}, the decay exponent gives a bound for the essential spectral gap. Following problem would be helpful in this quest as the IFSs arising from the convex co-compact Fuchsian groups are analytic:

\begin{problem}\label{prob:salem}
Suppose $\Phi$ is a $C^\omega$ IFS such that one of the maps $f_j \in \Phi$ is not a similitude. Is $K_\Phi$ is Salem set?
\end{problem}

Next cases beyond Problem \ref{prob:nonlinear} would be to study Fourier decay for problems that do not have uniform contraction such as systems that \textit{contract on average} (see e.g. the works of Werner from 2005 and Walkden \cite{Walkden} from 2013) or for \textit{parabolic} IFSs, that is, IFS where one of the branches has a parabolic fixed point: $f_j'(x) = 1$ for some point $x$ and $j \in \cA$, such as the inverse branches of the \textit{Manneville-Pommeau map} \cite{PS} or \textit{Farey map}. In such cases it is possible the large deviation estimates are no longer as strong (see e.g. the work of Pollicott and Sharp \cite{PS} from 2009) as the ones used in \cite{JordanSahlsten,SahlstenStevens}, which could cause issues in proving polynomial Fourier decay rate. Thus potentially the rate could be at most polylogarithmic in these cases if the typical orbits with respect to the self-conformal measure spend much time near the parabolic fixed point.

\begin{problem}\label{prob:parabolic}
If $\mu$ is a non-atomic self-conformal measure for $C^2$ IFS $\Phi$ such that one $f_j \in \Phi$ has a parabolic fixed point and $p_j > 0$. Then does the Fourier transform of $\mu$ decay at most polylogarithmically?
\end{problem}

A way to approach this could be to study uniformly contracting $C^2$ IFSs $\Phi_\eps$ that approach the parabolic IFS $\Phi$. In Theorem 1.6 in \cite{LeclercOsc} from 2022 Leclerc, demonstrated with some examples of non-linear IFS perturbations $\Phi_\eps$ of the self-similar IFS $\{x/2,x/2+1/2\}$ and a family of self-conformal measures $\mu_\eps$ with $|\widehat{\mu}_\eps(\xi)| \leq C_\eps |\xi|^{-\alpha}$ for $\Phi_\eps$ for a fixed $\eps$-independent exponent $\alpha > 0$, but the constants $C_\eps \to \infty$ as $\eps \to 0$ in front of the polynomial decay rate. The ideas here could be used to approach Problem \ref{prob:parabolic}. We will also discuss an analogue of Problem \ref{prob:parabolic} later in Section \ref{sec:higherdim} in the higher dimensional systems in the context of studying parabolic Julia sets.

Another direction to study Problem \ref{prob:parabolic} would be to \textit{accelerate} the dynamics, which would conjugate the parabolic IFS $\Phi$ to an infinite uniformly contracting IFS $\tilde \Phi$, and the self-conformal measure $\mu$ for $\Phi$ to a \textit{heavy tailed} Gibbs measure for $\tilde \Phi$. In \cite{JordanSahlsten} by Jordan and the author in 2016 infinite IFS coming from the Gauss map were considered, but the methods only works for \textit{light tailed} Gibbs measures due to the large deviation bound. This would need investigating the thermodynamical formalism for countable Markov shifts (see e.g. the work by Sarig \cite{Sarig} from 2003) with heavy tails.

\begin{problem}\label{prob:infinitebranches}
If $\mu$ is a Gibbs equilibrium state for a uniformly contracting IFS $\Phi$ with infinite branches and heavy tail, does the Fourier transform decay at most polylogarithmically?
\end{problem}

We note that in the infinite branch cases, uniform spectral gap for the transfer operator can be proved under some assumptions, see the work by Baladi and Vall\'ee \cite{BV} from 2005. If we would like to prove or disprove Problem \ref{prob:infinitebranches} as in \cite{JordanSahlsten,SahlstenStevens} one would also need to study large deviation theory for heavy tailed dynamical systems. The analogues for heavy tailed i.i.d. random processes in $\R$ e.g. for some sub-exponentially tailed walks were proved by Lehtomaa \cite{Lehtomaa} but the analogues for dynamical systems like countable Markov maps have not yet been developed. In \cite{BakerBanaji}, Baker and Banaji from 2023 are able to deal with certain stationary measures for infinite IFSs that are conjugated via a totally non-linear map to infinite linear IFSs (e.g. L\"uroth maps).

Now we have highlighted most of what is known currently in dimension $1$. Going to the higher dimensions there are advantages coming from the added degrees of freedom, but also new obstacles, which we will discuss in the next section.

\section{Higher dimensions}\label{sec:higherdim}

In dimensions $d \geq 2$ the Fourier transforms $\widehat{\mu}(\xi) = \int e^{-2\pi i \xi \cdot x} \, d\mu(x)$, $\xi \in \R^d$ for stationary measures $\mu$ for $\Phi = \{f_j : j \in \cA\}$ can have trivial obstructions due to the challenge on dealing with many directions in the frequencies $\xi \in \R^d$ at once. What conditions guarantee that the Fourier transform $|\widehat{\mu}(\xi)|$ decays (i.e. is a Rajchman measure) and what properties of the IFS control the rate?

We will first focus on the self-similar case, where all $f_j$ are similitudes of the form $r_j \theta_j + b_j$ for some rotation matrices $\theta_j \in \mathrm{SO}(d)$. If the rotations $\theta_j$ generate a wide enough subgroup of $\mathrm{SO}(d)$, they should help with gaining Fourier decay, even if the contraction is an inverse of a Pisot number. If $f_j$ are affine maps $A_j + b_j$ for some general matrices $A_j$ one has to focus on studying the random walks generated by the matrices and how they act on the sphere $\mathbb{S}^{d-1}$, which allows us to reduce much of the theory to that of Li \cite{Li2} from 2018. Finally, the general non-linear case is the least developed in this context, and has some intriguing connections to proving higher dimensional Fractal Uncertainty Principles \cite{Dyatlov} that has gained much attention recently (e.g. in Cohen's work \cite{Cohen} from 2023).

\subsection{Self-similar measures and random walks on rotations}

Let us take $d \geq 2$ and consider self-similar IFSs $\Phi$ on $\R^d$. This means that we can write
$$\Phi = \{f_j = r_j \theta_j + b_j : j \in \cA\}$$
for some $0 < r_j < 1$, $\theta_j \in \mathrm{SO}(d)$ and $b_j \in \R^d$. Let us now take a non-atomic self-similar measure associated to $\Phi$, that is, $\mu = \sum_{j \in \cA} p_j f_j(\mu)$ for some $0 < p_j < 1$ with $\sum_{j \in \cA} p_j = 1$. Note that if IFS $\Phi$ is affinely reducible, that is, the attractor $K_\Phi = \bigcup_{j \in \cA} f_j (K_\Phi)$ is contained in a proper non-trivial subspace $V$, then no self-similar measure for $\Phi$ can be Rajchman, because there is no Fourier decay along $\xi$ that are orthogonal to $V$. Thus the cases that are interesting happen when $\Phi$ is affinely irreducible.

Fourier transforms for self-similar measures were studied by Lindenstrauss and Varj\'u \cite{LV} in 2016 in their work of absolute continuity. Lindenstrauss and Varj\'u studied the absolute continuity of self-similar measures associated to IFSs that do not contract too much in $\R^d$ by connecting this problem to spectral information of transfer operators $T$  defined by the rotations $\theta_j$:
$$T f(\sigma) = \sum_{j \in \cA} p_j f(\theta_j^{-1} \sigma), \quad \sigma \in \mathrm{SO}(d), \quad f \in L^2(\mathrm{SO}(d)),$$
which, note, does \textit{not} depend on the translates $b_j$ of $f_j$. More precisely, \cite{LV} established that if the contractions $r_j$ are close enough to $1$ and the operator $T$ has a \textit{spectral gap}, i.e. $\|T^n\|_{L^2_0(\mathrm{SO}(d))} \to 0$ exponentially as $n \to \infty$, then $\mu$ is absolutely continuous. Here $L^2_0(\mathrm{SO}(d)) := \{f \in L^2(\mathrm{SO}(d)) : \int f = 0\}$.

The way absolute continuity of $\mu$ was established in \cite{LV}, a key part of the argument was to prove the Fourier transform of $\widehat{\mu}$ decays polynomially under the spectral gap of the operator $T$ and large contraction ratios. However, it is not clear what kind of assumptions lead to the spectral gap hypothesis on $T$. It is conjectured (see e.g. Benoist and de Saxc\'e \cite{BenoistSaxce} from 2016) that under dense rotations (i.e. $\theta_j$ form a dense subgroup), $T$ should have a spectral gap. However, this is only known when the rotations $\theta_j$ are \textit{algebraic}, in which case Bourgain and Gamburd \cite{BourgainGamburd} for $SO(3)$ in 2012, and Benoist and de Saxc\'e \cite{BenoistSaxce} in 2016 for general $d \geq 3$, proved the existence of spectral gap for $T$.

However, in 2013 in \cite{Varju3} Varj\'u in proved under dense rotations a weaker contraction property for $\|T^n\|_{L^2_0(\mathrm{SO}(d))}$ that quantitatively involves weights of irreducible representations of $\mathrm{SO}(d)$. This, together with the idea from \cite{LV}, can be used to prove subpolynomial Fourier decay for the associated self-similar measures. We will outline this here for equicontractive IFSs, and we remark that the proof likely generalises to general IFSs as did \cite{LV}. 

\begin{theorem}\label{thm:FourierSelfSimRd}
Let $\Gamma < \mathrm{SO}(d)$ with $d \geq 3$ be a dense subgroup generated by the rotations $\theta_j$, $j \in \cA$, and $\mu$ is a non-atomic self-similar measure associated to the IFS $\Phi = \{f_j = \lambda \theta_j + b_j\}$ with same contraction ratio $0 < \lambda < 1$ and that $f_j$ do not have a common fixed point, then there is a constant $c = c(\mu) > 0$ depending only on $\mu$ such that for any $\xi \in \R^d$, $|\xi| > 1$, we have
$$|\widehat{\mu}(\xi)| \lesssim e^{-c(\log |\xi|)^{1/3}}.$$
\end{theorem}

\begin{proof}[Proof of Theorem \ref{thm:FourierSelfSimRd}] The idea of the proof of Theorem \ref{thm:FourierSelfSimRd} was communicated to the author by P. Varj\'u.  Write $\xi = R z$ for $R > 1$ and $z \in \mathbb{S}^{d-1}$. Define the Fourier restriction $\widehat{\mu}_{R} : \mathbb{S}^{d-1} \to \C$,
$$\widehat{\mu}_{R}(z) := \widehat{\mu}(R z), \quad z \in \mathbb{S}^{d-1}.$$ 
Let $R_0 := \lambda^{n_0} R$ for some $n_0 \in \N$ we will choose later. Since $\mu = \sum_{j \in \cA} p_j f_j(\mu)$ and $f_j = \lambda \theta_j + b_j$, we have 
$$|\widehat{\mu}_R(z)| = \Big|\sum_{j \in \cA} p_j \int e^{-2\pi i (Rz) \cdot (\lambda \theta_j x + b_j) } \, d\mu(x)\Big| \leq \sum_{j \in \cA} p_j \Big|\int e^{-2\pi i (\lambda R\theta_j^{-1} z) \cdot x } \, d\mu(x) \Big| = T|\widehat{\mu}_{\lambda R}|(z).$$
Thus by iterating this $n_0$ times, we have
\begin{align}\label{eq:selfbound}|\widehat{\mu}(\xi)| = |\widehat{\mu}_R(z)| \leq T^{n_0} |\widehat{\mu}_{R_0}|(z).\end{align}
Note that here we identify $T$ also as an operator on $L^2(\mathbb{S}^{d-1})$ as follows: if $f \in L^2(\mathbb{S}^{d-1})$ and $z \in \mathbb{S}^{d-1}$, then write $Tf(z) := \sum_{j \in \cA} p_j f(\theta_j^{-1} z)$. Since for any $z \in \mathbb{S}^{d-1}$, we can always find $\sigma_z \in \mathrm{SO}(d)$, which maps $1$ to $z$, we have $Tf(z) = T \tilde f (\sigma_z)$, where $\tilde f(g) := f(g 1)$, $g \in \mathrm{SO}(d)$. Moreover, $\|\tilde f\|_{L^2(G)} = \|f\|_{L^2(\mathbb{S}^{d-1})}$ and 
$$\|T^k\tilde f\|_{L^2(G)} = \|T^kf\|_{L^2(\mathbb{S}^{d-1})}$$ 
for any $k \geq 1$.

Now, decompose $L^2(G)$ with respect to the Haar measure $dg$ on $G$ as a direct sum with irreducible representations $V_v < L^2(\mathrm{SO}(d))$ of weight $v$, that is, $L^2(\mathrm{SO}(d)) = \bigoplus_{v} V_v$, where $v=0$ corresponds to constant functions. Let $r \geq 1$, which we will choose later. Then, by Lemma 19 by Varj\'u \cite{Varju3}, there exists a function $h_r \in \bigoplus_{|v| \leq r} V_v$ such that $\int h_r(g) \, dg = 1$, $\|h_r\|_2 \lesssim r^{\dim \mathrm{SO}(d) / 4}$, where $\dim \mathrm{SO}(d) = \frac{d(d-1)}{2}$ and 
$$\int h_R(g) \dist(g,1) \, dg \lesssim_d r^{-1/2}.$$ 
Consider the convolution $|\widehat{\mu}_{R_0}| \ast h_r \in L^2(\mathbb{S}^{d-1})$, defined by
$$ |\widehat{\mu}_{R_0}| \ast h_r(z) := \int |\widehat{\mu}_{R_0}|(g^{-1} z) h_r(g) \, dg, \quad z \in \mathbb{S}^{d-1}.$$
Then we can write
\begin{align}\label{eq:selfbound2}|\widehat{\mu}_{R_0}|(z) =  \int (|\widehat{\mu}_{R_0}| \ast h_r)(v) \, dv + H(z) + F(z),\end{align}
where $H(z) = (|\widehat{\mu}_{R_0}| \ast h_r)(z) - \int (|\widehat{\mu}_{R_0}| \ast h_r)(v) \, dv$ and $F(z) = |\widehat{\mu}_{R_0}|(z) - |\widehat{\mu}_{R_0}| \ast h_r(z)$.  Combining \eqref{eq:selfbound2} with \eqref{eq:selfbound}, we have the bound
\begin{align}\label{eq:selfbound3} |\widehat{\mu}(\xi)| \leq \int (|\widehat{\mu}_{R_0}| \ast h_r)(v) \, dv  + |T^{n_0}H(z)| + |T^{n_0} F(z)|.\end{align}
To finish the argument, we need to bound the constant term and have an $L^\infty$ bound in $z$ of 
$$G(z) := |T^{n_0}H(z)| + |T^{n_0} F(z)|,$$ 
which is possible by an $L^2$ bound combined with a $\mathrm{Lip}$ norm bound, where the Lipschitz norm $\|G\|_{\mathrm{Lip}} := \|G'\|_\infty$.

For the constant term of \eqref{eq:selfbound3}, it was shown by Mattila \cite{Mattila3} (see Lemma 3.15 of \cite{Mattila95}) that for any Borel measure $\tau$ on $\R^d$, we have 
$$\int_{\mathbb{S}^{d-1}} |\widehat{\tau}(r z)|^2 \, dz \lesssim r^{-\delta} I_{\delta}(\tau),$$ 
where $I_\delta(\tau) = \iint |x-y|^{-\delta} \, d\mu(x) \, d\mu(y)$ for any $0 < \delta < (d-1)/2$. Now, for a self-similar measure $\mu$, we know that $I_{\eps_1}(\mu) < \infty$ for some $\eps_1 > 0$, e.g. following the argument of Feng and Lau \cite{FL} from 2009 as $\mu$ has no atoms. Thus by Cauchy-Schwarz
$$\Big| \int (|\widehat{\mu}_{R_0}| \ast h_r)(v) \, dv\Big|  \leq \|\widehat{\mu}_{R_0}\|_2 |\mathbb{S}^{d-1}|^{1/2} \|h_r\|_1 \lesssim R_0^{-\eps_1}.$$

By the $1$-Lipschitz continuity of the exponential, we have:
$$\Big\||\widehat{\mu}_{R_0}| - |\widehat{\mu}_{R_0}| \ast h_r \Big\|_2 \lesssim_d  R_0 r^{-1/2} \quad \text{and} \quad \Big\||\widehat{\mu}_{R_0}| - |\widehat{\mu}_{R_0}| \ast h_r\Big\|_{\mathrm{Lip}} \lesssim_d  R_0.$$
Since $\|H + F\|_{\mathrm{Lip}} \lesssim R_0$ and that $T$ does not increase $\|\cdot\|_{\mathrm{Lip}}$, we have $\|G\|_{\mathrm{Lip}} = \|T^{n_0}H + T^{n_0 }F\|_{\mathrm{Lip}} \lesssim R_0$.

For the $L^2$ norm of $G$, since 
$$\int h_R(g) \dist(g,1) \, dg \lesssim_d r^{-1/2},$$ 
using $1$-Lipschitz continuity we have $\|F\|_2 \lesssim R_0 r^{-1/2}$. Thus by triangle inequality 
$$\|T^{n_0} F \|_2 \leq \|F\|_2 \leq R_0 r^{-1/2}.$$

Next, as $h_r \in \bigoplus_{|v| \leq r} V_v$, $\int H = 0$, and as $H$ comes from a convolution with $h_r$, we can use the contraction of $T^n$ in $L^2_0(\mathrm{SO}(d))$ that depends on $r$, i.e. Theorem 6 of Varj\'u \cite{Varju3} from 2013 implying for some $c_0,c_0' > 0$ and $0 < A \leq 2$ such that
$$\|T^{n_0} H\|_2 \lesssim \Big(1-\frac{c_0}{ \log^{A} (c_0' r + 2)}\Big)^{n_0}\|H\|_2 \leq e^{-\frac{n_0 c_0}{ \log^{2} (c_0' r + 2)} }\|H\|_2.$$
Here by the Cauchy-Schwarz inequality, we have by the $L^2$ decay on spheres (like we argued for the constant term in \eqref{eq:selfbound3}) that $\|H\|_2 \lesssim R_0^{-\eps_1}$.  

Having bounds for $\|G\|_{\mathrm{Lip}}$ and $\|G\|_{2}$, we obtain $L^\infty$ bound:
$$\|G\|_\infty \lesssim \|G\|_{\mathrm{Lip}}^{1/3}  \|G\|_2^{2/3} \lesssim R_0^{1/3}  (e^{-\frac{n_0 c_0}{ \log^{2} (c_0' r + 2)} }R_0^{-\eps_1} + R_0 r^{-1/2})^{2/3}.$$
Choose now constant $c_1 > 0$ possibly depending on $\lambda$ such that when $R_0 = e^{c_1 (\log R)^{1/3}}$ (which also fixes $n_0$ as $\lambda^{n_0} R = R_0$) and $r = e^{10 c_1(\log R)^{1/3}}$ we have
$$R_0^{-\eps_1} + R_0^{1/3}  (e^{-\frac{n_0 c_0}{ \log^{2} (c_0' r + 2)} }R_0^{-\eps_1} + R_0 r^{-1/2})^{2/3} \leq e^{-c(\log R)^{1/3}}$$ 
for some constant $c > 0$. Thus the proof is complete.
\end{proof}

\begin{remark}Using a consequence of Theorem \ref{thm:FourierSelfSimRd} combined with a Davenport-Erd\"os-LeVeque criterion in $\R^d$ (similar proof to the one in $\R$ \cite{DEL}), one can prove that given an expanding toral endomorphism $A$, then $\{A^n x : n \in \N\}$ is equidistributed modulo $\mathbb{T}^d$ for $\mu$ almost every $x \in \R^d$ for any non-trivial self-similar measure associated to $\Phi$ with dense rotations. Compare this to the work of Dayan-Ganguly-Weiss \cite{DGW}, which established a similar result when the translations $b_j$ satisfy an irrationality condition using homogeneous dynamics. For the application of Theorem \ref{thm:FourierSelfSimRd} no irrationality condition is needed, but instead only on the rotation parts.
\end{remark}

Inspecting the proof of Theorem \ref{thm:FourierSelfSimRd} also leads to improved Fourier decay bound if we can guarantee more rapid contraction to $\|T^{n}\|$ as $n$ grows. This is possible if $T$ has a spectral gap, which would lead to a bound $\|T^{n_0} H\|_2 \lesssim \rho^{n_0}\|H\|_2$ in the proof of Theorem \ref{thm:FourierSelfSimRd} for some fixed constant $\rho > 0$, which allows one to get:

\begin{theorem}\label{thm:SpecGapFourierSelfSimRd}
Assume $d \geq 3$. Let $\mu$ be a self-similar measure on $\R^d$ for a probability vector $0 < p_j < 1$, $\sum_{j \in \cA} p_j = 1$ and the IFS $\Phi = \{f_j = \lambda \theta_j  + b_j\}$ with same contraction ratio $0 < \lambda < 1$ such that none of the pairs of the maps $f_j$ have a common fixed point. Suppose the transfer operator $T$ associated to the rotations $\{\theta_j : j \in \cA\} \subset \mathrm{SO}(d)$ has a spectral gap. Then there exists a constant $\alpha > 0$ such that
$$|\widehat{\mu}(\xi)| = O(|\xi|^{-\alpha}), \quad \text{as } |\xi| \to \infty.$$
\end{theorem}

As we discussed earlier, Theorem \ref{thm:SpecGapFourierSelfSimRd} was implicitely done in \cite{LV} as long as the contraction ratios are large enough. As with \cite{LV}, it is also likely that Theorem \ref{thm:SpecGapFourierSelfSimRd} holds for non-equicontractive IFSs. Thus if $\{\theta_j : j \in \cA\} \subset \mathrm{SO}(d)$, $d \geq 3$, form a dense subgroup, combined with the works of Bourgain, Gamburd, Benoist and de Saxc\'e \cite{BourgainGamburd,BenoistSaxce}, Theorem \ref{thm:SpecGapFourierSelfSimRd} implies polynomial Fourier decay for all non-trivial self-similar measures associated to the IFS $\Phi = \{f_j = \lambda \theta_j  + b_j\}$ as long as $f_j$ do not have a common fixed point.

Proving polynomial Fourier decay for self-similar measures in $\R^d$ in the general case is probably equally difficult as proving the spectral gap for $T$ under dense rotations:

\begin{problem}\label{prob:selfsimhigher}
Prove that if the rotations of the IFS $\Phi$ form a dense subgroup on $\mathrm{SO}(d)$, then all non-trivial self-similar measures associated to $\Phi$ have polynomial Fourier decay.
\end{problem}

As Problem \ref{prob:selfsimhigher} remains open, given the success of the parametrised approach in the Bernoulli convolution and general self-similar measures case, one could ask if the polynomial Fourier decay for self-similar measures would hold for \textit{almost every} choice of parameters? 

\begin{problem}\label{prob:selfsimhigher2}
Prove that for Haar$^{\cA}$ almost every choice of rotations $(\theta_j : j \in \cA) \in \mathrm{SO}(d)^\cA$, non-trivial self-similar measures on $\R^d$ are Rajchman measures with polynomial Fourier decay.
\end{problem}

Now comparing to the results of Erd\H{o}s and Salem (Theorem \ref{thm:erdossalem}), and Li and the author \cite{JialunSahlsten1} (Theorem \ref{thm:lisahlsten}) and Br\'emont \cite{Bremont} (Theorem \ref{thm:bremont}), would it be possible to provide an algebraic characterisation of self-similar measures in a similar fashion? This was addressed by Rapaport \cite{Rapaport}. Therefore, a possible characterisation is interesting in the affinely irreducible case, which Rapaport considered in \cite{Rapaport} in 2022:

\begin{theorem}\label{thm:rapaport}
Any self-similar measure $\mu$ on $\R^d$ associated to an affinely irreducible IFS $\Phi = \{r_j \theta_j + b_j : j \in \cA\}$ is Rajchman if and only if the following condition does not hold: there exists a subspace $V \subset \R^d$ with $d' = \dim V > 0$ such that $\theta_j (V) = V$ for all $j \in \cA$ and an isometry $S : V \to\R^{d'}$ such that
\begin{itemize}
\item[(1)] If $U_j' \in O(d')$ and $a_j' \in \R^{d'}$ are defined such that 
$$S \circ \pi_V \circ \phi_j \circ S^{-1} = r_j U_j' + b_j'$$ 
where $\pi_V$ is the orthogonal projection onto $V$, and $H \subset \mathrm{GL}(d',\R)$ is the subgroup generated by matrices $r_j U_j'$, $j \in \cA$, then $N := H \cap O(d')$ is finite normal subgroup of $H$ and $H / N$ is cyclic.
\item[(2)] If $A \in H$ is contraction such that $\{A^n N : n\in \Z\} = H / N$, then there exists $k \geq 1$, $z_1,\dots,z_k \in \C$ and $\zeta_1,\dots,\zeta_k \in \C^{d'} \setminus \{0\}$ such that
\begin{itemize}
\item[(i)] $\{z_1,\dots,z_k\}$ is a P.V. $k$-tuple, that is, $z_i$, $i = 1,\dots,k$, are distinct complex numbers outside the closed unit disk that are roots of an integral monic polynomial whose other roots are all inside the open unit disk;
\item[(ii)] $A^{-1} \zeta_i = z_i \zeta_i$ for $1 \leq i \leq k$;
\item[(iii)] if $j \in \cA$ and $V \in N$, there is a polynomial $P_{j,V}$ of rational coefficients such that $\langle V b_j' , \zeta_i \rangle = P_{j,V}(z_i)$ for $1 \leq i \leq k$.
\end{itemize}
\end{itemize}
\end{theorem}

This gives for example a corollary that any non-Rajchman measure $\mu$ for an affinely irreducible self-similar IFS $\Phi = \{r_j \theta_j + b_j : j \in \cA\}$ implies every contraction is of the form $r_j = \lambda^{n_j}$ where $n_j \geq 1$ is an integer and $\lambda^{-1} > 1$ is an algebraic integer as a member of a P.V. $k$-tuple.

If we go beyond the self-similar case to self-affine systems that are properly different from self-similar ones, it turns out that the analogous spectral gap bounds can be achieved. This will allow us to overcome the issues in the self-similar case and prove polynomial Fourier decay in these settings. However, algebraic characterisations like the one by Rapaport in Theorem \ref{thm:rapaport} become more unclear. These will be the theme of the next section.

\subsection{Self-affine measures and random walks on affine groups}

Let us now consider a general self-affine IFS $\Phi$, that is, we have
$$\Phi = \{f_j = A_j + b_j : j \in \cA\}$$
where $A_j \in \mathrm{GL}(d,\R)$, $\|A_j\| < 1$ for all $j \in \cA$ and self-affine measure $\mu = \sum_{j \in \cA} p_j f_j(\mu)$ for some $0 < p_j < 1$ with $\sum_{j \in \cA} p_j = 1$. Write $\Gamma$ as the subgroup generated by the linear parts $A_j$, $j \in \cA$. If $\Gamma$ is \textit{proximal}, that is, we can find $A \in \Gamma$ such that $A$ has only one eigenvalue $\lambda_A$ such that $\lambda_A$ is simple and $|\lambda_A|$ is the largest amongst all the eigenvalues of $A$. If proximality fails, the closure $\overline{\Gamma}$ is compact, so we can conjugate $\Phi$ to a self-similar IFS on $\R^d$. In particular, the study of Fourier transforms of self-affine measures for non-proximal $\Phi$ reduces to the self-similar case.

In the proximal case, the \textit{irreducibility} properties of $\Gamma$ become essential in the study of the Fourier decay. We say that $\Gamma$ is \textit{totally irreducible} if one cannot construct a finite number of non-empty proper subspaces of $\R^d$ such that $\Gamma$ would fix them. Adapting a renewal theoretic approach as in the self-similar case \cite{JialunSahlsten1}, Li and the author \cite{JialunSahlsten2} in 2019 proved the Rajchman property for self-affine measures in this case:

\begin{theorem}\label{thm:LiSahlstenRd}
Suppose the subgroup $\Gamma$ generated by the linear parts $A_j$, $j \in \cA$, of a self-affine IFS $\Phi = \{A_j + b_j : j \in \cA\}$ is proximal and totally irreducible. Then the Fourier transform of every non-trivial self-affine measure $\mu$ associated to $\Phi$ is Rajchman.
\end{theorem}

Theorem \ref{thm:LiSahlstenRd} was heavily motivated by Li's earlier work \cite{Li2} from 2018 on random walks in the split semisimple Lie groups, where Li studied the Fourier decay properties of stationary measures associated to the random walks. Here we needed to study the renewal theorem for the random walks acting on the sphere $\mathbb{S}^{d-1}$ so one had to extend the renewal theorems to such settings.

To get polynomial Fourier decay, we need a rate in the renewal theorem for random walks on $\mathbb{S}^{d-1}$, where, as in Li's work \cite{Li2}, we analyse $\Gamma$ more precisely in the \textit{Zariski topology}. Zariski topology is a natural topology for subsets of $\mathrm{GL}(d,\R)$ where closed subsets are algebraic subvarieties in a finite dimensional vector space, see \cite[Section 5]{BQ} for a definition. In \cite{Li2}, the following result was obtained in the general setting that gives a condition for polynomial Fourier decay of self-affine measures using the Zariski closure of $\Gamma$:

\begin{theorem}\label{thm:LiSahlstenRdpower}
Suppose the Zariski closure of $\Gamma$ is a connected and $\R$-splitting reductive group acting irreducibly on $\R^d$, where $\Gamma$ is the subgroup generated by the linear parts $A_j$, $j \in \cA$, of a self-affine IFS $\Phi = \{f_j = A_j + b_j : j \in \cA\}$ such that any pair of the maps $f_j$ do not have a common fixed point. Then the Fourier transform of every non-trivial self-affine measure $\mu$ associated to $\Phi$ is Rajchman with polynomial Fourier decay.
\end{theorem}

See Definition 3.6 in \cite{Li2} for the definition of $\R$-splitting reductive groups. The proof to achieve polynomial Fourier decay in Theorem \ref{thm:LiSahlstenRdpower} reduces by the renewal theoretic argument to having \textit{exponential error} rate in the renewal theorem for random walks on $\mathbb{S}^{d-1}$. Proving this error rate in Theorem \ref{thm:LiSahlstenRdpower}, as in Li's work for random walks on reductive Lie groups, is ultimately based on a uniform spectral gap theorem for these random walks. To then prove these uniform spectral gap results, one needs a polynomial Fourier decay result for stationary measures for the random walks, which in turn are based on generalising the argument by Bourgain and Dyatlov \cite{BD1} done by Li \cite{Li2} by adapting ideas due to Dolgopyat \cite{Dolgopyat,Dolgopyat1} from 1998. Here in particular one needs a higher dimensional version of Theorem \ref{thm:discretised}, which replaces the non-concentration condition by a projective version proved by Li in \cite{JialunRd} from 2018 for the ring $\R^d$ with the product given by  $ab = (a_1b_1,\dots,a_d b_d)$, $a,b \in \R^d$:

\begin{theorem}\label{thm:discretisedRd}
For all $\kappa > 0$, there exists $\eps_1,\eps_2 > 0$ and $k \in \N$ such that the following holds. Let $\mu$ be any probability measure on $[1/2,1]^d$ and $N \in \N$ large enough. Assume $\mu(x : v\cdot x \in B(y,\rho)) \lesssim \rho^{\kappa}$ for all $\rho \in [1/N,1/N^{\eps_1}]$, $y \in \R$ and $v \in \mathbb{S}^{d-1}$. Then for all $\xi \in \R$, $|\xi| \approx N$:
$$\Big|\int e^{2\pi i \xi \cdot x_1 x_2 \dots x_k} \, d\mu(x_1) \dots d\mu(x_k)\Big| \lesssim |\xi|^{-\eps_2}.$$
\end{theorem}

The proof of Theorem \ref{thm:discretisedRd} relied on a sum-product theorem on real Lie groups by He and de Saxc\'e \cite{HdS18} from 2018. Thus the $\R$-split condition of the Zariski closure of $\Gamma$ is essential in checking the projective non-concentration condition. In dimensions $2$ and $3$, this $\R$-split condition of the Zariski closure is unnecessary, and we only need irreducibility of $\Gamma$ and non-compactness of the image of $\Gamma$ in $\mathrm{PGL}(d,\R)$ to guarantee polynomial decay. For example, in dimension $2$, these conditions always imply Zariski density of $\Gamma$. When we go to dimensions $d \geq 4$, one can construct examples that demonstrate that non-compactness of the image is not enough \cite[Remark 1.2]{JialunSahlsten2}, so it remains an interesting open problem to classify self-affine Rajchman measures:

\begin{problem}\label{prob:selfaffinecharpoly}
Provide an algebraic characterisation for self-affine iterated function systems that lead to every self-affine measure to be Rajchman with polynomial Fourier decay.
\end{problem}

For example, in dimension $2$, would irreducibility of $\Gamma$ and non-compactness of the image of $\Gamma$ in $\mathrm{PGL}(2,\R)$ characterise polynomial Fourier decay of all self-affine measures associated to $\Phi$? We also remark that Li's work \cite{Li2} on proving polynomial Fourier decay of stationary measures for split Lie groups also has links to other works such as in the study of exponential mixing, see \cite{LiPan} and also \cite{LPX}, where it was used to prove non-concentration.

Now, given Rapaport's algebraic classification of Rajchman self-similar measures \cite{Rapaport}, perhaps an easier direction towards Problem \ref{prob:selfaffinecharpoly} would be to only  characterise the Rajchman property. This is a problem suggested by Ariel Rapaport:

\begin{problem}\label{prob:selfaffinechar}
Provide an algebraic characterisation for self-affine iterated function systems that lead to every non-trivial self-affine measure to be Rajchman.
\end{problem}

As we discussed earlier, if the group $\Gamma$ is not proximal, then Problem \ref{prob:selfaffinechar} reduces essentially to the self-similar case and thus the work of Rapaport \cite{Rapaport} can be used. Therefore, to achieve a characterisation of the Rajchman property, we would need to consider the case of proximal but not totally irreducible $\Gamma$. This is essentially the \textit{solvable} case and an attractive case would be to initially study dimension $d = 2$ and the diagonal case when $\Gamma$ is isomorphic to $\R^2$. In general in dimension $2$, one could predict that if the closure of $\Gamma$ is not contained in a one-dimensional subgroup, one would still have a Rajchman property.

Motivated also by the polynomial Fourier decay for \textit{almost every} Bernoulli convolutions and self-similar measures in $\R$, one could also ask if self-affine measures have something like this. In the diagonal case, which is thus not covered by \cite{JialunSahlsten2}, Solomyak \cite{Solomyak3} in 2022 did recently study this and obtained:

\begin{theorem}\label{thm:soloyakRd}
For Lebesgue almost every $(\lambda_1,\dots,\lambda_d) \in (0,1)^{d}$ if $\mu$ is a non-trivial self-affine measure for the self-affine IFS $\Phi_{A} = \{f_j = A + b_j : j \in \cA\}$, where $A$ is the diagonal matrix with diagonal $(\lambda_1,\dots,\lambda_d)$ such that any pairs of the maps $f_j$ do not have a common fixed point, then $\mu$ has polynomial Fourier decay.
\end{theorem}

Classifying self-affine Rajchman measures and what conditions lead to polynomial Fourier decay is useful information in the study of general $C^{1+\alpha}$ IFSs in $\R^d$. For example, given the advances towards Problem \ref{prob:nonlinear}, any obstructions in the self-affine case to Fourier decay, could also happen in the non-linear case if a conjugacy to such self-affine system exists. In the next section we will review what is known in the general case in $\R^d$.

\subsection{Non-linear $C^{1+\alpha}$ systems in $\R^d$ and Fractal Uncertainty Principles}

The Fourier transforms of stationary measures for higher dimensional $C^{1+\alpha}$ IFSs $\Phi = \{f_j : j \in \cA\}$ are not yet that widely studied as the previous cases, but there are some recent developments we will discuss. When studying $C^{1+\alpha}$ IFSs, one strategy could be to try to separate the study to conformal situation, that is, when $f_j'(x) = r_j(x) \theta_j(x) + b_j$ for some orthogonal matrix $\theta_j(x)$ in $\R^d$, in which case we could try to exploit either the self-similar case in $\R^d$, and then to the study of more general situations such as quasiconformal maps. 

We also note that spectral gaps for transfer operators that play a major role in proving non-concentration estimates needed for polynomial Fourier decay, do have their analogues in higher dimensions already, e.g. in the study of exponential mixing of Anosov flows \cite{Stoyanov} by Stoyanov from 2011 and the Teichm\"uller flow \cite{AGY} by Avila, Gou\"ezel and Yoccoz from 2005. Hence it is plausible that Fourier decay results, such as the ones done in dimension $1$, have their analogues in $\R^d$. A potential problem to tackle here could be:

\begin{problem}\label{prob:nonlinearRd}
Suppose $\Phi = \{f_j : j \in \cA\}$ is a $C^{1+\alpha}$ IFS in $\R^d$ such that one of the maps $f_j$ is not a similitude and the attractor $K_\Phi$ is not contained in a proper subspace of $\R^d$. Does the Fourier transform of every self-conformal measure giving positive weight to every branch associated to $\Phi$ decay polynomially?
\end{problem}

Some advances towards the setting of Problem \ref{prob:nonlinearRd} have been done recently. In 2019 in \cite{JialunNaudPan} Li, Naud and Pan studied the Patterson-Sullivan measures of limit sets of Kleinian groups, where a polynomial Fourier decay theorem was proved. 

\begin{theorem}\label{thm:LiNaudPan}
Let $\Gamma$ be a Zariski dense Kleinian Schottky subgroup of $\mathrm{PSL}(2,\C)$. Then the Fourier transform $\widehat{\mu}_\Gamma$ of the Patterson-Sullivan measure $\mu_\Gamma$ associated to $\Gamma$ has polynomial Fourier decay.
\end{theorem}

Here the difficulty was in establishing the non-concentration condition. By relating the non-concentration to certain twisted transfer operators studied by Oh-Winter \cite{OW} from 2017, Leclerc \cite{Leclerc} in 2022 adapted the strategy of \cite{SahlstenStevens} to the Fourier decay of natural measures arising from Julia sets $J_f$ for hyperbolic rational map $f : \widehat{\C} \to \widehat{\C}$, where $\widehat{\C}$ is the Riemann sphere. Examples include $f_c(z) = z^2 + c$ when $c \in (0,1/4)$, when $f_c$ is hyperbolic and the Julia set $J_{f_c}$ is a quasicircle, and $c \in (1/4,1)$ in which case the Julia set $J_{f_c}$ is a Cantor set.

\begin{theorem}\label{thm:Leclerc}
Let $f : \widehat{\C} \to \widehat{\C}$ be a hyperbolic rational map of degree $d \geq 2$. If the associated $J_f \subset \C$ is not contained in a circle, and $\mu_\phi$ is any stationary measure associated to a potential $\phi : C^1(V,\R)$ on an open neighbourhood $V$ of $J_f$, then $\widehat{\mu}_\phi$ has polynomial Fourier decay at infinity.
\end{theorem}

This applied in particular to the conformal measure and the measure of maximal entropy. Here Leclerc was working in the hyperbolic regime. Now, comparing to Problem \ref{prob:parabolic} in the one-dimensional case, it would be interesting to study what happens to the Fourier decay rate when the Julia set $J_f$ becomes close to something that is known to be Salem or when the map approaches something that might fail to have this.  For example, if $f_c(z) = z^2+c$ and $J_c$ is the Julia set of $f_c$, then the Julia set $J_c$ approaches to a circle as $c \to 0$, which is a Salem set (length measure has decay rate $|\xi|^{-1/2}$, see the book by Mattila \cite{Mattila2}) as we discussed in the introduction. For this J. Fraser posed the following problem:

\begin{problem}\label{prob:julialimit}
Does there exist probability measures $\mu_c$ on $J_c$ with 
$$|\widehat{\mu}_c(\xi)| \lesssim |\xi|^{-\alpha_c},$$
as $|\xi| \to \infty$, such that $\alpha_c \to \frac{1}{2}$ as $c \to 0$?
\end{problem}

On the other hand, at $c=1/4$ there is a phenomenon called \textit{parabolic implosion} and the dynamics $f_{1/4}$ is parabolic. What would happen for Fourier decay rates when $c \to 1/4$? Note the Hausdorff dimension $c \mapsto \Hd J_{f_c}$ is discontinuous at $c = \frac{1}{4}$. An interesting open problem here suggested by F. Naud is to study the following problem:

\begin{problem}\label{prob:parabolicJulia}
What happens to the polynomial Fourier decay rates of stationary measures for $f_c$ when $c \to 1/4$?
\end{problem}

To adapt ideas of Leclerc \cite{Leclerc}, one would need to develop the work of Oh and Winter \cite{OW} in a parabolic context, which was suggested to the author by G. Leclerc. 

Now to attack more general IFSs, recently Khalil \cite{Khalil} in 2023 showed that if a measure $\mu$ does not concentrate near subspaces, then there is a polynomial decay of the Fourier transform except along a very sparse set of frequencies, generalising a similar statement for self-similar measures on $\R$ by Tsujii:

\begin{theorem}\label{thm:KhalilRd}
Let $\mu$ be a compactly supported Borel probability measure on $\R^d$, which is $(C,\alpha)$ uniformly affinely non-concentrated for some $C,\alpha > 1$, that is, for any $\rho > 0$, $x \in \R^d$, $0 < r \leq 1$ and any affine hyperplane $V$ in $\R^d$, we have
$$\mu(\{y \in B(x,r) : \mathrm{dist}(y,V) \leq \rho r \} ) \leq C\rho^\alpha \mu(B(x,r)).$$
Then for any $\eps > 0$ there exists $\delta > 0$ such that for all $T > 0$:
$$\mathrm{Leb}(\{\xi \in B(0,T) : |\widehat{\mu}(\xi)| > T^{-\delta}\}) = O_\eps(T^\eps).$$
\end{theorem}

Khalil established Theorem \ref{thm:KhalilRd} by proving a new higher dimensional analogue of Shmerkin's $L^q$ flattening theorem under additive convolutions in $\R$ \cite{Shmerkin} and adapting the $L^2$ case. Theorem \ref{thm:KhalilRd} applies to self-conformal measures associated to self-conformal iterated function systems in $\R^d$ whose attractors are not contained in a hyperplane \cite{RS} providing evidence towards a possible solution to Problem \ref{prob:nonlinearRd}.

We could also compare Khalil's work to the recent paper by Dasu and Demeter \cite{DD} from 2022. There Dasu and Demeter studied Frostman measures on $\R^2$ supported on curved graphs, and were able to give $L^6$ information on the Fourier transform of such measures:

\begin{theorem}\label{thm:Demeter}
Suppose $\Gamma$ is a graph of a $C^3$ function with non-zero curvature everywhere. If $\mu$ is a Borel measure on $\Gamma$ satisfying $\mu(B(x,r)) \lesssim r^{\delta}$, $x \in \R^2$ and $r > 0$ for some constant $\delta > 0$, then there exists $\beta = \beta(\delta) > 0$ such that for any ball $B_R$ of radius $R > 0$ we have
$$\|\widehat{\mu}\|_{L^6(B_R)} \lesssim R^{\frac{2-2\delta - \beta}{6}}.$$
\end{theorem}

The methods behind Theorem \ref{thm:Demeter} rely on the notion of \textit{decoupling} that exploits the non-linear nature of the graph function $\Gamma$. It would be interesting to try to link the methods of decoupling to the study of more general non-linear $C^{1+\alpha}$ IFSs.

Parallel to the study of Fourier decay of stationary measures for IFSs, there has also been activity in the community of quantum chaos to try to study \textit{Fractal Uncertainty Principles} in $\R^d$ \cite{Dyatlov,DJ,BD2}. Analogously as with defined in $\R$ in the earlier section, two sets $X,Y \subset \R^d$ are said to satisfy the \textit{Fractal Uncertainty Principle} (FUP) with exponent $\beta > 0$ if for all $0 < h < 1$, whenever $f \in L^2(\R^d)$ satisfies $\supp \widehat{f} \subset \{h^{-1} y : \mathrm{dist}(y,Y) \leq h\}$, then 
$$\int_{X(h)} |f(x)|^2 \, dx \lesssim h^{2\beta}.$$ 
If $X = \R \times \{0\}$ and $Y = \{0\} \times \R$, then FUP does not hold for these sets in $\R^2$, which is a similar obstruction as for the decay of Fourier transform of measures in $\R^2$. Thus there seems to be very similar obstructions to proving FUP in $\R^d$ as with getting polynomial Fourier decay for stationary measures. Moreover, polynomial Fourier decay can be used to prove Fractal Uncertainty Principle for the supports of the measures in some regimes \cite{BakerSahlsten}. 

Adapting Dolgopyat's method, Backus, Leng and Tao \cite{BLT} in 2023 proved FUP for Ahlfors-David regular subsets that are \textit{non-orthogonal with respect to the dot product}  \cite[Definition 1.2]{BLT} meaning roughly that $X,Y$ are not contained in submanifolds that have orthogonal tangent lines.

\begin{theorem}\label{thm:BLT}
If $X,Y \subset \R^d$ are Ahlfors-David regular sets with dimensions $\delta_X$ and $\delta_Y$ respectively, which are non-orthogonal with respect to the dot product, then $X$ and $Y$ satisfy the FUP with exponent $\beta = \frac{d}{2} - \frac{\delta_X+\delta_Y}{2} + \alpha$ for some $\alpha > 0$.
\end{theorem}

The result of Backus, Leng and Tao also applied for more general phase functions, which allowed them to apply the result to bounding the essential spectral gap by a work of Dyatlov and Zahl \cite{DyatlovZahl} from 2016. However, due to the exponent $\beta = \frac{d}{2} - \frac{\delta_X+\delta_Y}{2} + \alpha$, this limited the potential fractals to sets where $\frac{\delta_X+\delta_Y}{2}$ is not too large. This case was considered by Cohen \cite{Cohen} in 2023, who proved FUP in the case of \textit{line $\rho$-porous sets $Y$ from scales $1$ to $h^{-1}$} (here $0 < h < 1$), that is, sets $Y$ for which for all line segment $\tau$ of length $1 \leq r \leq h^{-1}$, we have $B(y,\rho r) \cap Y = \emptyset$ for some $y \in \tau$. If $X$ is assumed to be just \textit{$\rho$-porous on scales $h$ to $1$}, that is, for any $h \leq r \leq 1$ and $x \in X$, there exists $y \in \R^d$ such that $B(y,\rho r) \subset B(x,r) \setminus X$. In this case Cohen was able to prove the FUP when the frequency set is line porous, generalising the work of Bourgain and Dyatlov \cite{BD2} from 2018:

\begin{theorem}\label{thm:CohenRd}
If $X \subset [-1,1]^d$ is $\rho$-porous for scales $[h,1]$ and $Y \subset [-h^{-1},h^{-1}]^d$ is line $\rho$-porous from scales $1$ to $h^{-1}$. Then there exists $\beta = \beta(\rho,d) > 0$ such that $X$ and $Y$ satisfy FUP with the exponent $\beta$.
\end{theorem}

It would be interesting to check which $C^{1+\alpha}$ IFSs produce attractors that satisfy the line porosity condition. This was studied by Chousionis \cite{Chousionis} from 2009, where Chousionis proved in the conformal case the attractor $K_\Phi$ is always \textit{directionally $m$-porous} for all $m$-dimensional subspaces $V$ of $\R^d$ as long as $K_\Phi$ is purely $m$-unrectifiable. Inspecting the methods behind Chousionis and Cohen could be helpful in trying to approach Problem \ref{prob:nonlinearRd} in its full generality.

Next we will discuss some potential future directions beyond the cases we discussed above.

\section{Prospects}\label{sec:prospects}

Leclerc \cite[Theorem 1.6]{LeclercOsc} proved in 2022 that for certain stationary measures $\mu_\eps$ for an IFS $\Phi_\eps$ that is an $\eps$-perturbations of the IFS $\Phi = \{x/2,x/2+1/2\}$ for any $\eps > 0$, one has $|\widehat{\mu}_\eps(\xi)| \lesssim \eps^{-1}|\xi|^{-\alpha}$ as $|\xi|\to \infty$ for some $\alpha > 0$ that is \textit{independent} of $\eps$. As $\eps \to 0$, these measures $\mu_\eps$ will weakly converge to a singular Bernoulli measure $\mu$ for $\Phi$, which as $\times 2$ invariant singular measure is not Rajchman. Thus, Leclerc's result demonstrates that the polynomial Fourier decay rate behaves discontinuously in general in perturbations. However, as the bound was of the form $|\widehat{\mu}_\eps(\xi)| \lesssim \eps^{-1}|\xi|^{-\alpha}$, where the linear coefficient $\eps^{-1}$ will blow-up as $\eps \to 0$, it would be interesting instead just to focus on the linear factors in these bounds and how they behave as $\eps \to 0$. This would lead to a notion of \textit{Fourier complexity} of Rajchman measures with polynomial Fourier decay, where we ignore the polynomial decay term and study lower order term. This could give more finer geometric information on the fractal measure $\mu$.

In 2023 Fraser \cite{Fraser} introduced a notion called the \textit{Fourier spectrum}, which interpolates between the Hausdorff dimension and Fourier dimension of compact sets $K \subset \R^d$. In particular, for sets such as the middle third Cantor set, where there is no Fourier decay for any measure supported on them, the notion allows to get Fourier information on $K$ by reducing some small portion of frequencies depending on an interpolating parameter $p$. The Fourier spectrum is not yet well studied for stationary measures of IFSs, and it could be interesting to see how sensitive it is to the geometric information that plays a major role in the study of Fourier dimension. 

Much of the theory developed here is for IFSs on $\R^d$. What about other spaces and geometries? For example, the Heisenberg group or other sub-Riemannian manifolds, or, as suggested by J. Li in an email communication from 2023, one could try to study Fourier decay in finite fields $\mathbb{F}_p$ or general non-archimedean fields. Studying analogues of the geometric measure theoretic problems in finite fields (e.g. the Kakeya problem) has seen much attention over the recent years. In $\mathbb{F}_p$ there there is not yet much theory done for Fourier transforms of IFSs. Here one could study subsets of the finite field that exhibit some form of self-similarity, and study how this influences the Fourier coefficients of measures on these sets. As with some advances done before the discretised sum-product theorems in finite fields, ideas here could also help with the study of their $\R^d$ analogues, like the hard questions about self-similar measures and a toy model to study the Salem property of attractors to IFSs.

Finally, we can see with many of the examples that the ambient dimension of a stationary measure $\mu$ plays a major role in potentially obstructing the decay of the Fourier transform. Motivated by limiting the directions where Fourier decay happens, in a recent work Leclerc \cite{LeclercBunched} from 2023 introduced a notion of Fourier dimension on manifolds. This notion shares ideas with the semiclassical \textit{wavefront sets} (see the book \cite{Zworski} by Zworski), where one studies directions that do not enjoy rapid Fourier decay. Given the links to the Fractal Uncertainty Principle surveyed in \cite{Dyatlov} and discussed in Section \ref{sec:higherdim}, developing this theory could be interesting in connecting to the challenging study of eigenfunctions of the Laplacian on manifolds (e.g. Quantum Unique Ergodicity conjecture by Rudnick and Sarnak \cite{RS94} from 1994, see the works of Dyatlov and Jin \cite{DJ} from 2018 and \cite{DJN} by Dyatlov, Jin and Nonnenmacher from 2022 for recent developments), which correspond to the analogues of the waves $e^{2\pi i \xi \cdot x}$ used to define $\widehat{\mu}(\xi)$ on $\R^d$.

\section*{Acknowledgements}

We thank Amir Algom, Simon Baker, Amlan Banaji, Semyon Dyatlov, Jonathan Fraser, Gaétan Leclerc, F\'elix Lequen, Jialun Li, Michael Magee, Pertti Mattila, Tuomas Orponen, Ariel Rapaport, Nicolas de Saxc\'e, Boris Solomyak, Connor Stevens, Lauritz Streck, Petri Vesanen, Joe Thomas, P\'eter Varj\'u, Sanju Velani, Caroline Wormell and Zhiyuan Zhang for useful discussions and comments during the preparation of this manuscript. We also thank the anonymous referee for many helpful suggestions to improve the paper. T.S. is supported by the Research Council of Finland's Academy Research Fellowship \emph{``Quantum chaos of large and many body systems''}, grant Nos. 347365, 353738.

\bibliographystyle{plain}

\vspace{0.1in}

\end{document}